\documentclass[a4paper,12pt]{article}

\newenvironment{proof}{\noindent {\bf Proof:}}{\hfill $\Box$}

\newtheorem{theorem}{Theorem}
\newtheorem{lemma}{Lemma}

\newtheorem{remark}{Remark}

\usepackage{booktabs}
\usepackage{color}

\usepackage{algorithm}
\usepackage[noend]{algpseudocode}
\makeatletter
\def\BState{\State\hskip-\ALG@thistlm}
\makeatother
\usepackage{algorithmicx}

\usepackage{amsmath}
\usepackage{amssymb}
\usepackage{amsfonts}

\usepackage[normalem]{ulem}
\usepackage[dvipsnames]{xcolor}

\usepackage{tikz}

\newcommand{\bs}{\boldsymbol}
\newcommand{\mr}[1]{\mathrm{#1}}

\newcommand{\D}{\mathcal{D}}

\newcommand{\Mc}{\mathcal{M}}

\newcommand{\Kc}{\mathcal{K}}

\newcommand{\Dc}{\mathcal{D}}

\newcommand{\Yf}{\mathbf{Y}}
\newcommand{\Uf}{\mathbf{U}}
\newcommand{\Zf}{\mathbf{Z}}
\newcommand{\Kf}{\mathbf{K}}
\newcommand{\Xf}{\mathbf{X}}

\newcommand{\Rb}{\mathbb{R}}
\newcommand{\Nb}{\mathbb{N}}

\newcommand{\Ind}{\mathbb{I}}
\newcommand \cel[2]{\{#1,\ldots,#2\}}
\newcommand*\interior[1]{#1^{\mathsf{o}}}

\title{\bf Moments and convex optimization for analysis and control of nonlinear partial differential equations}

\begin{document}

\author{Milan Korda$^1$, Didier Henrion$^{2,3,4}$, Jean Bernard Lasserre$^2$}

\footnotetext[1]{University of California, Santa Barbara,\; {\tt milan.korda@engineering.ucsb.edu}}
\footnotetext[2]{CNRS; LAAS; 7 avenue du colonel Roche, F-31400 Toulouse; France. {\tt henrion@laas.fr},  {\tt lasserre@laas.fr}}
\footnotetext[3]{Universit\'e de Toulouse; LAAS; F-31400 Toulouse; France.}
\footnotetext[4]{Faculty of Electrical Engineering, Czech Technical University in Prague,
Technick\'a 2, CZ-16626 Prague, Czech Republic.}

\date{ \today}

\maketitle

\begin{abstract}
This work presents a convex-optimization-based framework for analysis and control of nonlinear partial differential equations. The approach uses a particular weak embedding of the nonlinear PDE, resulting in a \emph{linear} equation in the space of Borel measures. This equation is then used as a constraint of an infinite-dimensional linear programming problem (LP). This LP is then approximated by a hierarchy of convex, finite-dimensional, semidefinite programming problems (SDPs). In the case of analysis of uncontrolled PDEs, the solutions to these SDPs provide  bounds on a specified, possibly nonlinear, functional of the solutions to the PDE; in the case of PDE control, the solutions to these SDPs provide bounds on the optimal value of a given optimal control problem as well as  suboptimal feedback controllers. The entire approach is based purely on convex optimization and does not rely on spatio-temporal gridding, even though the PDE addressed can be fully nonlinear. The approach is applicable to a very broad class nonlinear PDEs with polynomial data.  Computational complexity is analyzed and several complexity reduction procedures are described. Numerical examples demonstrate the approach.
\end{abstract}

\begin{flushleft}\small
\end{flushleft}

\section{Introduction}
This paper considers the problem of analysis and optimal control of nonlinear partial differential equations (PDEs) with polynomial data. The approach proceeds by embedding the nonlinear PDE into an infinite-dimensional space of Borel measures, which results in a \emph{linear} equation involving the so called occupation and boundary measures as variables. This equation, along with a user-specified objective functional, gives rise to an infinite-dimensional linear programming problem (LP) in the space of Borel measures. In the uncontrolled case, the solution to this infinite-dimensional LP provides a bound on a given functional (e.g., energy) of  the solution to the PDE. In the controlled case, the solution to this LP provides a bound on the optimal value of the optimal control problem considered.

This infinite-dimensional LP is then approximated using the classical Lasserre's hierarchy of finite-dimensional \emph{convex} semidefinite programming (SDP) relaxations. In the uncontrolled case, this hierarchy provides a monotonous sequence of upper or lower bounds on the optimum of the infinite dimensional LP and hence also upper and lower bounds on the given functional of the solution to the PDE. In the controlled case, the solutions to the SDPs provide bounds on the optimal value of the optimal control problem considered as well as a sequence of suboptimal polynomial feedback controllers. The sequence of SDP relaxations is proven to converge to the optimal value of the infinite dimensional LP as the relaxation degree tends to infinity, under standard assumptions. The SDP relaxations can be readily solved by standard semidefinite programming solvers such as MOSEK or SeDuMi~\cite{sedumi}.

The major advantage of the presented approach is its overall convexity, thereby avoiding the inherent issues associated to non-convex optimization such as local minima. This comes at the expense of having to solve a sequence of convex semidefinite programming problems  of increasing size, governed by the number of unknown functions in the PDE plus the number of spatio-temporal dimensions of the PDE and the number of derivatives appearing nonlinearly. In order for the approach to be applicable to large-scale problems such as those coming from computational fluid dynamics (e.g., the Navier-Stokes equations in three spatial dimensions), structure/sparsity exploitation (e.g. using~\cite{waki2006sums}) or tailored moment-sum-of-squares solvers (e.g.,~\cite{papp2017sum}) may be required. In this paper, we briefly describe several of the most obvious ways for structure exploitation and complexity reduction, although a more detailed analysis of the problem structure is likely to uncover more.

At his point we would like to stress that the goal of this work is not to present a new PDE \emph{solver} but rather a numerical procedure, based on convex optimization, for obtaining 
\emph{bounds} on functionals of the solutions to the PDE and for \emph{control} design. Besides being of independent theoretical interest, such bounds are important in robustness analysis (e.g., to verify that the maximum force on a given construction does not exceed the breakage point of the material used) as well as for validation of numerical solvers. The use of the control design portion of our work is broad, spanning virtually all applications where feedback control of PDEs is required as well as the field of PDE-constrained optimization (e.g.,\cite{hinze2008optimization}), where the ``control'' may not necessarily represent a feedback control in the classical sense but, for example, a parametrization of the shape of a material to be designed.

The distinguishing and important feature is that the proposed approach does not rely on spatio-temporal gridding but rather works with moment sequences of measures supported on the graphs of the solution to the PDE. In fact, the bounds or the polynomial feedback controller are obtained without ever numerically solving the PDE. This is especially useful in the case of complex dynamical behavior (e.g., a high Reynolds number fluid) where fine gridding is necessary to ensure convergence of the numerical scheme for solving the PDE, thereby making the solution very costly. On the other hand, the size of the SDP solved by our approach is not affected by the complexity of the dynamical behavior studied, once the number of moments is fixed. The quality of the bound or controller obtained with a given number of moments may, of course, depend on the complexity of the dynamical behavior. Nevertheless, the bounds obtained are always valid, irrespective of the number of moments considered, up to the numerical precision of the SDP solution.

It remains an open theoretical question under what conditions the infinite dimensional LP is equivalent to the original problem. This question was resolved for the problems of optimal control of ordinary differential equations in~\cite{vinter1993convex} and later for the problems of region of attraction and invariant set computation in~\cite{kordaROA} and \cite{kordaMCI}. The situation appears to be far more complex for nonlinear PDEs; in the uncontrolled case, one pathway to establishing the equivalence may lead through the weak-strong uniqueness results of, e.g., \cite{brenier2011weak}, which suggest expanding the set of constraints of the infinite dimensional LP with additional  inequalities satisfied by any strong solution.

The presented approach can be seen as a generalization of the approach of~\cite{lasserre2008nonlinear} from nonlinear ordinary differential equations to nonlinear partial differential equations.

To the best of the authors' knowledge, this is the first time that nonlinear PDEs are addressed using convex semidefinite programming at this level of generality. Besides the concurrent work \cite{burgers} focusing on hyperbolic conservation laws with the emphasis on convergence aspects, we are aware of only a few other attempts in the technical literature. For example, the early work~\cite{rubio1995global} exploited the linear measure formulation of non-linear (semi-linear elliptic) PDE control problems, although without convergence proofs and using a computationally inefficient linear programming discretization. The moment-SOS hierarchy was used systematically in \cite{mevissen2008solving} for solving approximately non-linear PDEs, formulated after time and domain discretization as large-scale sparse non-convex polynomial optimization problems. Bounds on functionals of solutions were obtained with SOS polynomials for linear elliptic PDEs in~\cite{kawai2015explicit}, for non-linear PDEs arising in fluid dynamics in~\cite{chernyshenko2014polynomial} and for the non-linear Kuramoto-Sivashinsky PDE in~\cite{goluskin2018bounds}. These works, however, provide bounds with no convergence guarantees and do not exploit the ``primal'' formulation of the problem on measures, which we believe to be crucial for convergence analysis. In the context of classical linear matrix inequality (LMI) Lyapunov approaches, there have been recently several works that developed LMI conditions for PDE analysis and controller design, but only in the linear case; see e.g. \cite{fridman2009lmi,barreau2018lyapunov,lamare2016optimisation,valmorbida2016stability,gahlawat2017convex} and references therein. Finally, occupation measures and the moment-SOS hierarchy were used in \cite{magron2017optimal} for Riesz-spectral linear PDE control by state-mode discretization and in \cite{bertsimas2006bounds} for obtaining bounds on functionals of solutions to linear PDEs. Our contribution with respect to these references are as follows:
\begin{itemize}
\item we consider non-linear PDE analysis and control;
\item we do not use time, space or spectral discretization\footnote{The approach used in this paper bears some similarity to the classical Galerkin discretization since the PDE is enforced (in the weak form) on a certain subspace of smooth test functions (in our case on all polynomials up to a given degree). However, contrary to the Galerkin method, the solution to the PDE is \emph{not} sought as a linear combination of these test functions but rather represented by the moments of the associated occupation measure whose positivity is enforced using semidefinite programming.} of the PDEs;
\item we introduce the notion of occupation and boundary measures in the PDE context;
\item we provide rigorous convergence proofs for the moment-SOS hierarchy in the PDE context.
\end{itemize}

The paper is organized as follows. Section~\ref{sec:statement_anal} describes the problem setting for the analysis of PDEs. Section~\ref{sec:occupMeas} introduces the occupation and boundary measures and uses them to derive a linear embedding of the nonlinear PDE, resulting in an infinite-dimensional LP. Section~\ref{sec:SDPuncont} presents the hierarchy of SDP relaxations providing the upper and lower bounds and describes several generalizations of the approach. Section~\ref{sec:statement_cont} describes the problem setup for PDE control. Section~\ref{sec:linRepCont} describes the linear embedding for the controlled case, leading an infinite-dimensional LP. Section~\ref{sec:SDPcont} presents the SDP relaxations for the controlled case and a way to extract a feedback controller from its solutions. Section~\ref{sec:higherOrder} describes a generalization to higher-order PDEs. Section~\ref{sec:complex} discusses the computational complexity of the approach and presents several complexity reduction techniques. Section~\ref{sec:numEx} presents numerical examples and we conclude and given an outlook in Section~\ref{sec:conclusion}.

\section{Problem statement (analysis)}\label{sec:statement_anal}
We consider the system of nonlinear partial differential equations\footnote{The summation indices in~(\ref{eq:pde}) run through $i\le n, j\le n$, $i\le j$ (due to the symmetry of the Hessian matrix); we omit the summation bounds for conciseness.} (PDEs)
\begin{subequations}\label{eq:pdeEntire}
\begin{align}
F(x,y(x),\D y(x)) + \sum_{i,j} B_{i,j}(x,y(x)) \frac{\partial ^2 y}{\partial x_i\partial x_j} &= 0 \;\;\; \mr{in}\;\; \interior\Omega \label{eq:pde} \\ 
G(x,y(x),\D y(x)) &= 0  \;\;\; \mr{on}\;\; \partial\Omega, \label{eq:bnd}
\end{align}
\end{subequations}
where $y:\Rb^n\to \Rb^{n_y}$ is the unknown, possibly vector-valued, function defined on a domain $\Omega\subset \mathbb{R}^n$ and $F:\Rb^n\times \Rb^{n_y}\times \Rb^{n_y\times n}\to \Rb^{n_F}$ and $G:\Rb^n\times\Rb^{n_y}\times \Rb^{n_y\times n}\to \Rb^{n_G}$ are given, possibly vector-valued, functions. The function $F$ is assumed to be a multivariate polynomial in its arguments whereas $G$ is allowed to be piecewise polynomial\footnote{Allowing $G$ to be piecewise polynomial provides a significant modeling freedom for the boundary condition~(\ref{eq:bnd}). For example, setting $G = 0$ on a certain piece of the boundary results in the trivial constraint $0 = 0$, thereby not assigning any boundary condition on this piece (which is commonly encountered in PDEs involving space and time where the initial condition is typically assigned whereas the terminal condition is free or vice versa).}. The symbol $\D y$ denotes the Jacobian matrix of~$y$, i.e., $(\D y)_{ij} = \frac{\partial y_i}{\partial x_j}$, where $\frac{\partial y_i}{\partial x_j}$ denotes the weak derivative of $y_i$ with respect to $x_j$ (in the Sobolev spaces sense). The coefficients $B_{i,j}(x,y(x)):\Rb^n\times \Rb^{n_y} \to\Rb^{n_F\times n_y}$ are fixed matrix functions with each entry being a polynomial in $(x,y)$ and \[\frac{\partial ^2 y}{\partial x_i\partial x_j} = \left[\frac{\partial ^2 y_1}{\partial x_i\partial x_j},\ldots, \frac{\partial ^2 y_{n_y}}{\partial x_i\partial x_j}\right]^\top, \]
where $\cdot^\top$ denotes the transpose. The domain $\Omega$ is assumed to be a compact basic semialgebraic\footnote{A basic semialgebraic set is the intersection of the super or sub-level sets of finitely many multivariate polynomials. This is a broad class of sets covering most of the domains considered in practice such as boxes, spheres, ellipsoids and intersections thereof.} set of the form
\begin{equation}\label{eq:omegaDef}
\Omega = \{x \in \Rb^n \mid g^\Omega_i(x)\ge 0,\; i = 1,\ldots,n^\Omega\},
\end{equation}
where $g_i^\Omega$ are polynomials and $n^\Omega \in \Nb$. 
 The set $\Omega$ is assumed to be equal to the closure of its interior $\interior\Omega$ and its boundary $\partial \Omega$ is assumed to be (locally) Lipschitz (in the sense of~\cite[Definition 12.9]{leoni2017first}).
 

Since $\Omega$ is basis semialgebraic, its boundary is piecewise smooth and hence admits a surface measure $\sigma$ satisfying the integration by parts formula (a special case of the Stokes formula)
\begin{equation}\label{eq:stokes}
\int_{\partial\Omega} h(x)\eta_m(x)\,d\sigma(x) = \int_{\Omega} \frac{\partial h}{\partial x_m}(x)\,dx,\;\; m=1,\ldots,n
\end{equation}
for any $h \in W^{1,\infty}(\interior\Omega)$, where $\eta(x) = (\eta_1(x),\ldots,\eta_n(x))$ is the outward unit surface normal vector to $\partial \Omega$. The boundary $\partial \Omega$ can be decomposed as
\begin{equation}\label{eq:boundPartiotion}
\partial \Omega = \bigcup_{i=1}^{n_{\mr{b}}} \Omega_{\partial,i}
\end{equation}
with each $\Omega_{\partial,i}$ being basic semialgebraic and smooth, i.e.,
\begin{equation}\label{eq:partialOmega_i}
\Omega_{\partial,i} = \{x \in \Rb^n \mid g^{\partial,i}_j(x)\ge 0,\; j = 1,\ldots,n_{\partial,i}\},
\end{equation}
where $g^{\partial,i}_j$ are polynomials and $n_{\partial,i}\in \Nb$. The partition~(\ref{eq:boundPartiotion}) can be taken disjoint up to a set of zero surface measure $\sigma$ (e.g., corner points). The piecewise polynomial function $G$ is assumed to be polynomial on each element $\Omega_{\partial,i}$ of the partition.

\begin{remark}[Weak solutions, regularity and uniqueness]\label{rem:regul}
Throughout this paper we consider solutions $y(\cdot)$ to the PDE (\ref{eq:pdeEntire}) to lie in the Sobolev space $W^{k,\infty}$ with $k =1$ if $B_{i,j} = 0$ and $k = 2$ otherwise. The equations ~(\ref{eq:pde}) and~(\ref{eq:bnd}) are assumed to be satisfied almost everywhere with respect to the Lebesgue measure on $\interior\Omega$ respectively the surface measure $\sigma$ on $\partial\Omega$. In some cases, this notion of solution is too wide to be of practical interest; if this the case, the equations~(\ref{eq:pde}) and (\ref{eq:bnd}) must be supplemented with additional conditions that single out the solutions of interest. These conditions may, for example, be the semiconcavity for Hamilton-Jacobi type PDEs or the entropy inequalities for conservation laws (see, e.g., Sections 3.3.3 and 3.4.3 of~\cite{evansPDEbook}). We note that the $W^{k,\infty}$ regularity assumption is made in order to avoid technicalities; all results presented in this paper hold also with $y(\cdot) \in W^{k,p}$ with sufficiently large $p$ depending on the degree of the polynomials $F$ and $G$.
\end{remark}

This is a very general setting encompassing a large number of equations of interest, including the compressible as well as incompressible Navier-Stokes equations.

\begin{remark}[Higher-order PDEs]
Higher order nonlinear PDEs can be handled using the proposed approach by introducing additional variables. This is described in Section~\ref{sec:higherOrder}.
\end{remark}

 For the time being, we consider the problem of analysis of such a PDE (control is treated in Section~\ref{sec:statement_cont}); in particular we are interested in establishing bounds on a functional of the form
\begin{equation}\label{eq:obj}
J\big(y(\cdot)\big) := \int_{\Omega}L\left(x,y(x),\D y(x)\right)\,dx + \int_{\partial\Omega}L_\partial\left(x,y(x),\D y(x)\right)\,d\sigma(x),
\end{equation}
where $L:\Rb^n\times \Rb^{n_y}\times \Rb^{n_y\times n}\to \Rb$ is a given multivariate polynomial and $L_\partial:\Rb^n\times \Rb^{n_y}\times \Rb^{n_y\times n}\to \Rb$ is a piecewise multivariate polynomial such that each restriction of $L_{\partial}$ to $\Omega_{\partial,i}$ is a multivariate polynomial. This is a fairly general form of a functional; for example, if $y$ represents the velocity field of a fluid, then~(\ref{eq:obj}) with $L(x,y,z) = \sum_i y_i^2$ and $L_\partial = 0$ is proportional to the average kinetic energy of the fluid.

Now we are ready to state the first goal of this paper:

\textbf{Problem 1}\quad {\it Find upper and lower bounds on~(\ref{eq:obj}) evaluated along the solutions of~(\ref{eq:pdeEntire}).}

Formally speaking, this task can be expressed as finding upper and lower bounds on the optimal value of the optimization problem
\begin{equation}\label{opt:ocp_uncont}
\begin{array}{ll}
\underset{y(\cdot)}{\mbox{inf}\,/\,\mbox{sup}} & J\big(y(\cdot)\big)  \\
\mbox{subject to} &  (\ref{eq:pde}),\; (\ref{eq:bnd}).
\end{array}
\end{equation}
If minimization is considered in~(\ref{opt:ocp_uncont}), then lower bounds on the optimal value of this optimization problem yield lower bounds on~(\ref{eq:obj}) evaluated along the solutions of~(\ref{eq:pde}) and (\ref{eq:bnd}). Similarly, if maximization is considered in~(\ref{opt:ocp_uncont}), then upper bounds on the optimal value of~(\ref{opt:ocp_uncont}) yield upper bounds on~(\ref{eq:obj}).

We note that, since the PDE (\ref{eq:pde}) and (\ref{eq:bnd}) is nonlinear, the problem~(\ref{opt:ocp_uncont}) is an infinite-dimensional \emph{nonconvex} optimization problem. One of the main contributions of this work is a method of finding  bounds on the optimal value of this  infinite-dimensional nonconvex problem in terms of the solution to a finite-dimensional \emph{convex} optimization problem.

\begin{remark}[Uniqueness]
We make no a priori assumptions on the uniqueness of the solutions to~(\ref{eq:pdeEntire}). If~(\ref{eq:pdeEntire}) has a unique solution, then the feasible set of the optimization problem~(\ref{opt:ocp_uncont}) is a singleton in which case the infimum and supremum of~(\ref{opt:ocp_uncont}) coincide; otherwise, if the solution to~(\ref{eq:pdeEntire}) is not unique, the infimum and supremum may differ. In either case, the approach presented provides upper bounds on the supremum and lower bounds on the infimum. See also Remark~\ref{rem:regul}.
\end{remark}

\begin{remark}[Non-polynomial Dirichlet boundary conditions]
The presented approach can be extended to non-polynomial Dirichlet boundary conditions of the form $y(x) = h_i(x)$, $x\in\partial\Omega_i$, $i \in \mathcal{I}$, $\mathcal{I}\subset \{1,\ldots,n_{\mr{b}}\}$, where $h_i$ is an arbitrary integrable function such that the integrals $\int_{\partial\Omega_i} x^\alpha h_i^\beta(x)\,d\sigma(x)$ are computable (either analytically or numerically). This is treated in Section~\ref{sec:nonPolBnd}. Note that this extension can also be used with polynomial $h_i$, leading to a simpler representation than when using the general form~(\ref{eq:bnd}).
\end{remark}

\begin{remark}[Periodic boundary conditions]
The presented approach can also be extended handle periodic boundary conditions. This  is treated in Section~\ref{sec:periodicBnd}.
\end{remark}

\section{Occupation measures for nonlinear PDEs}\label{sec:occupMeas}
The goal of this section is to transform the nonconvex optimization problem~(\ref{opt:ocp_uncont}) to an infinite-dimensional \emph{linear} (and hence convex) programming problem. The key ingredient for doing so are the so-called occupation and boundary measures associated to the solutions of the PDE~(\ref{eq:pdeEntire}). For notational simplicity we assume that the solutions to~(\ref{eq:pdeEntire}) satisfy \[
y(x) \in \Yf\subset \Rb^{n_y},\quad \D y(x) \in \Zf\subset \Rb^{n_y\times n}.
\]
 Given any solution to (\ref{eq:pdeEntire}), the \emph{occupation measure} associated to $y(\cdot)$ is defined by
\begin{equation}\label{eq:occupMeas}
\mu(A\times B\times C) = \int_{\Omega} \Ind_{A\times B\times C}\left(x,y(x),\D y(x)\right)\,dx
\end{equation}
for all Borel sets $A\subset \Omega$, $B\subset \Yf$, $C\subset \Zf$. Similarly, the \emph{boundary measure} is defined by
\begin{equation}\label{eq:boundaryMeas}
\mu_\partial(A\times B\times C) = \int_{\partial\Omega} \Ind_{A\times B\times C}\left(x,y(x),\D y(x)\right)\,d\sigma(x),
\end{equation}
for all Borel sets $A\subset \partial\Omega$, $B\subset \Yf$, $C\subset \Zf$. The occupation measure captures the behavior of the solution $y(\cdot)$ and its first derivative in $\Omega$ whereas $\mu_\partial$ captures the behavior on the boundary. The measures $\mu$ and $\mu_\partial$ are nonnegative Borel measures on $\Omega\times \Yf\times\Zf$ respectively $\partial\Omega\times \Yf\times\Zf$. It follows immediately from the definition of $\mu$ and $\mu_\partial$ that for any bounded Borel measurable function $h:\Omega\times \Yf\times \Zf \to \Rb$ we have
\begin{subequations}
\begin{equation}\label{eq:occupMeasFun}
\int_{\Omega}  h\left(x,y(x),\D y(x)\right)\,dx = \int_{\Omega\times \Yf\times \Zf}  h\left(x,y,z\right)\,d\mu(x,y,z)
\end{equation}
\text{and}
\begin{equation}\label{eq:boundaryMeasFun}
\int_{\partial\Omega}  h\left(x,y(x),\D y(x)\right)\,d\sigma(x) = \int_{\partial\Omega\times \Yf\times \Zf}  h\left(x,y,z\right)\,d\mu_\partial(x,y,z).
\end{equation}
\end{subequations}
We note that on the right-hand side of~(\ref{eq:occupMeasFun}) and (\ref{eq:boundaryMeasFun}), $y$ and $z$ are not functions of $x$ anymore but rather integrated variables. Therefore, the measures $\mu$ and $\mu_\partial$ allow us to evaluate integral functionals of the solutions to the PDE~(\ref{eq:pdeEntire}) using integration with respect to  ordinary nonnegative Borel measures defined on the Euclidean space subsets $\Omega\times \Yf\times \Zf$ respectively $\partial\Omega\times \Yf\times \Zf$. This simple observation will turn out to be crucial in constructing computable bounds on the optimal value of~(\ref{opt:ocp_uncont}).

\subsection{Linear representation}\label{sec:linRep}
Now we will use the occupation and boundary measures to derive a \emph{linear} equation in the space of Borel measures which is satisfied by any pair of occupation and boundary measures arising from a solution to~(\ref{eq:pdeEntire}).

\subsubsection{Scalar case}
To begin, we will demonstrate the derivation in the scalar case ($n=n_y = 1$) and with $B_{i,j} = 0$. To this end, let $y(\cdot)$ be a solution to~(\ref{eq:pdeEntire}) and denote $z(x) := \frac{\partial y}{\partial x}(x)$. Given any test function $\phi\in C^\infty( \Omega \times \Yf )$ we have using~(\ref{eq:stokes})
\[
 \int_{\partial \Omega} \phi(x,y(x))\eta(x) d\sigma(x)  = \int_{\Omega}\frac{d}{dx}\phi(x,y(x))\, dx = \int_{\Omega} \frac{\partial \phi}{\partial x} + \frac{\partial \phi}{\partial y} \frac{\partial y}{\partial x}\, dx = \int_{\Omega} \frac{\partial \phi}{\partial x} + \frac{\partial \phi}{\partial y} z(x)\, dx,
\]
where $\eta(x)$ is the outward unit surface normal vector to $\partial \Omega$. Using~(\ref{eq:occupMeasFun}) and (\ref{eq:boundaryMeasFun}), we get
\begin{equation}\label{eq:aux1}
  \int_{\partial \Omega\times \Yf\times \Zf} \phi(x,y)\eta(x)\,d\mu_\partial(x,y,z) - \int_{\Omega\times \Yf\times \Zf} \frac{\partial \phi}{\partial x} + \frac{\partial \phi}{\partial y} z\,d\mu(x,y,z) = 0.
\end{equation}
Similarly, we also have
\begin{align*}
&\int_{\Omega}\phi(x,y(x)) F(x,y(x),z(x))\,dx = 0, \\
& \int_{\partial\Omega}\phi(x,y(x)) G(x,y(x),z(x))\,d\sigma(x) = 0.
\end{align*}
 Using~(\ref{eq:occupMeasFun}) and (\ref{eq:boundaryMeasFun})  these equations translate to
\begin{align}
&\int_{\Omega\times \Xf \times \Zf} \phi(x,y)F(x,y,z)\,d\mu(x,y,z) = 0 \label{eq:aux2} \\
&\int_{\partial\Omega\times \Xf \times \Zf} \phi(x,y) G(x,y,z)\,d\mu_\partial(x,y,z) = 0. \label{eq:aux3}
\end{align}
Notice that in~(\ref{eq:aux1}), (\ref{eq:aux2}), (\ref{eq:aux3}), $y$ and $z$ are no longer functions of $x$ but rather integrated variables. In addition, crucially, we observe that these equations depend \emph{linearly} on $(\mu,\mu_\partial)$. Therefore, we have shown that any solution to the nonlinear PDE~(\ref{eq:pdeEntire})  satisfies the system of \emph{linear} equations~(\ref{eq:aux1}), (\ref{eq:aux2}), (\ref{eq:aux3}) in variables $(\mu,\mu_\partial)$ for all test functions $\phi\in C^\infty(\Omega\times \Yf)$.

\subsubsection{General case}
We have the following theorem:
\begin{theorem}\label{thm:linearRepGen}
Let $y(\cdot) \in W^{k,\infty}(\interior\Omega;{\bf Y})\cap C(\Omega;{\bf Y}) $ ($k=1$ if $B_{i,j} = 0$, $k=2$ otherwise) satisfy~(\ref{eq:pde}) and (\ref{eq:bnd}) almost everywhere with respect to the Lebesgue measure on $\interior\Omega$ respectively the surface measure $\sigma$ on $\partial\Omega$. If $\mu$ and $\mu_\partial$ are defined by~(\ref{eq:occupMeas}) and (\ref{eq:boundaryMeas}), then for all $\phi \in C^\infty(\Omega\times \Yf)$
\begin{subequations}\label{eq:thmEqEntire}
\begin{align}\label{eq:thmEq1}
\int_{\Omega\times \Yf\times \Zf} \left[ \frac{\partial \phi}{\partial x_m} + \sum_{k=1}^{n_y} \frac{\partial \phi}{\partial y_k}z_{k,m}\right]d\mu(x,y,z) -  \int_{\partial\Omega\times \Yf\times \Zf} \phi(x,y) \eta_m(x)\,d\mu_\partial(x,y,z) = 0
\end{align}
for all $m\in \{1,\ldots,n\}$ and
\begin{align}\label{eq:thmEq2}
\int_{\Omega\times \Yf\times \Zf} \phi(x,y)F(x,y,z)\,d\mu(x,y,z) &\nonumber\\&\hspace{-40mm} - \sum_{i,j}\int_{\Omega\times \Yf\times \Zf}  \phi(x,y)\left[\frac{\partial B_{i,j}}{\partial x_j}+ \sum_{k=1}^{n_y}\frac{\partial B_{i,j}}{\partial y_k} z_{k,j} \right] z_{\bullet,i}\,d\mu(x,y,z)\nonumber \\ & \hspace{-40mm}  - \sum_{i,j}\int_{\Omega\times\Xf\times\Yf}   \left[\frac{\partial \phi}{\partial x_j} +\sum_{k=1}^{n_y} \frac{\partial \phi}{\partial y_k}z_{k,j}  \right]B_{i,j}(x,y) z_{\bullet,i}\,d\mu(x,y,z)\nonumber\\& \hspace{-40mm} + \sum_{i,j}\int_{\partial\Omega\times\Yf\times\Zf} B_{i,j}(x,y)\phi(x,y)z_{\bullet,i}\eta_j(x)\,d\mu_\partial(x,y,z) = 0
\end{align}
and
\begin{equation}\label{eq:thmEq3}
\int_{\partial\Omega\times \Xf \times \Zf} \phi(x,y) G(x,y,z)\,d\mu_\partial(x,y,z) = 0,
\end{equation}
where $z_{\bullet,i}$ denotes the $i^{\mr{th}}$ column of the matrix variable $z \in \Rb^{n_y\times n}$.
\end{subequations}
\end{theorem}
\begin{proof}
To prove (\ref{eq:thmEq1}), fix $m \in \{1,\ldots,n\}$ and $\phi \in C^\infty(\Omega\times \Yf)$. Since $y$ is continuous and $y \in W^{k,\infty}$ with $k\in\{1,2\}$ and $\partial\Omega$ is Lipschitz, $y$ is in fact Lipschitz on~$\Omega$ (this follows from~\cite[Theorem~4, p. 294]{evansPDEbook} and the extension theorem \cite[Theorem 12.15]{leoni2017first} used in the last step of the proof). Consequently, $\phi(x,y(x))$ is also Lipschitz on $\Omega$ and therefore the integration by parts formula~(\ref{eq:stokes}) can be applied to obtain
\[
\int_{\partial \Omega} \phi(x,y(x)) \eta_m(x)\,d\sigma(x) = \int_{\Omega}\left [ \frac{\partial \phi}{\partial x_m} + \sum_{k=1}^{n_y} \frac{\partial \phi}{\partial y_k}\frac{\partial y_k}{\partial x_m}\right] dx.
\]
Using~(\ref{eq:occupMeasFun}) and (\ref{eq:boundaryMeasFun}), this is equivalent to
\[
\int_{\partial \Omega\times \Xf\times\Zf} \phi(x,y) \eta_m(x)\,d\mu_\partial(x,y,z) = \int_{\Omega\times \Xf\times\Zf}\left [ \frac{\partial \phi}{\partial x_m} + \sum_{k=1}^{n_y} \frac{\partial \phi}{\partial y_k}z_{k,m}\right] d\mu(x,y,z),
\]
which is~(\ref{eq:thmEq1}).

To prove (\ref{eq:thmEq2}), fix  $\phi \in C^\infty(\Omega\times \Yf)$ and assume $y\in W^{2,\infty}(\interior\Omega ; \Yf)$ satisfies (\ref{eq:pde}) for almost all $x \in \Omega$. Then the function $\phi(\cdot,y(\cdot))$ is in $W^{2,\infty}(\Omega ;\Rb)$. In addition, since $F$ and $B_{i,j}$ are polynomials in all its arguments, each component of the vector-valued function $F(\cdot,y(\cdot),\Dc y(\cdot))$ is in $L_\infty(\Omega;\Rb)$ and each component of the matrix-valued function $B_{i,j}(\cdot,y(\cdot)) $ is in $W^{2,\infty}(\Omega ; \Rb)$. Therefore, each term on the left-hand side of (\ref{eq:pde}) is in $L_\infty(\Omega)$ and since~(\ref{eq:pde}) holds for almost all $x\in\interior\Omega$ and $\Omega$ is compact, we have
\[
\int_{\Omega} \phi(x,y(x))  F\left(x,y(x),\D y(x)\right)\,dx + \sum_{i,j} \int_{\Omega}\phi(x,y(x)) B_{i,j}(x,y(x)) \frac{\partial ^2 y}{\partial x_i\partial x_j} \, dx = 0.
\]
The equation~(\ref{eq:thmEq2}) follows by using~(\ref{eq:occupMeasFun}) on the first term of the above expression and (\ref{eq:occupMeasFun}) and (\ref{eq:boundaryMeasFun}) on the second term after using the integration by parts formula~(\ref{eq:stokes}) and the chain rule, which are justified under the regularity conditions stated above since both $\phi(\cdot,y(\cdot))B_{i,j}(\cdot,y(\cdot))$ and $\frac{\partial y}{\partial x_j}$ are Lipschitz.

Finally, equation~(\ref{eq:thmEq3}) follows immediately by integrating~(\ref{eq:bnd}) against a test function $\phi\in C^\infty(\Omega\times \Yf)$ and using~(\ref{eq:boundaryMeasFun}).
\end{proof}

We note that the assumption $y(\cdot) \in W^{k,\infty}$ is made in order to avoid technicalities and can be relaxed; see Remark~\ref{rem:regul}.


More importantly, we note that as in the scalar case, the equations~(\ref{eq:thmEq1}), (\ref{eq:thmEq2}), (\ref{eq:thmEq3}) depend \emph{linearly} on $(\mu,\mu_\partial)$. We also note that $x,y,z$ are all integrated  variables; in particular $y$ and $z$ are \emph{not} a function of $x$ in these equations. Notice also that if $F$, $G$ and $B_{i,j}$ are polynomial, all integrands in~(\ref{eq:thmEqEntire}) are polynomials in $(x,y,z)$ provided that the test function $\phi(x,y)$ is polynomial and provided that the normal vector $\eta(x)$ depends polynomially\footnote{The assumption of $\eta$ being polynomial is made just for the simplicity of exposition. In section~\ref{sec:normalVec} we describe how to relax this assumption.} on~$x$. This will be crucial in developing a tractable finite-dimensional approximation of these infinite-dimensional equations.

Note also that equations (\ref{eq:thmEq1})-(\ref{eq:thmEq3}) are in general \emph{relaxations} in the sense that the set of all measures $(\mu,\mu_\partial)$ satisfying these equations may be larger than the set of all occupation and boundary measure pairs corresponding to the solutions to the PDE~(\ref{eq:pdeEntire}). Note in particular that the set of all measures satisfying (\ref{eq:thmEq1})-(\ref{eq:thmEq3}) is convex, whereas the set of all occupation and boundary measure pairs satisfying~(\ref{eq:pdeEntire}) may be nonconvex (unless the solution is unique, which we do not assume). Therefore, in the very least, the set of all solutions to  (\ref{eq:thmEq1})-(\ref{eq:thmEq3}) contains the entire closed \emph{convex hull} of the set of all boundary and occupation measure pairs. It remains an open question whether the solution set of~(\ref{eq:thmEq1})-(\ref{eq:thmEq3}) in fact coincides with this convex hull. This question is not only of independent theoretical interest but also a prerequisite for tightness of the bounds developed in the following sections; see Remark~\ref{rem:tightBounds} for a further discussion.

\subsection{Infinite-dimensional linear program}
Now we are ready to write down an infinite-dimensional LP whose optimal value will provide a lower or upper bound on the optimal value of~(\ref{opt:ocp_uncont}). We denote by $\Mc(A)_+$ the set of all nonnegative Borel measures with supports included in the set $A$. We shall decompose the boundary measure $\mu_{\partial}$ according to the partition of the boundary of~$\Omega$~(\ref{eq:boundPartiotion}), i.e., we write
 \begin{equation}\label{eq:boundMeasDecomp}
\mu_{\partial} = \sum_{i=1}^{n_\mr{b}}\mu_{\partial_i},\quad \mu_{\partial_i}\in \Mc(\partial\Omega_i\times \Yf\times \Zf)_+. 
 \end{equation}
 We note that the decomposition~(\ref{eq:boundMeasDecomp}) is unique since $\mu_{\partial_i}(\partial \Omega_j\times \Yf\times \Zf) = 0$ whenever $i\ne j$; this follows from the fact that, by~(\ref{eq:boundaryMeas}), the $x$-marginal of $\mu_\partial$ is equal to the surface measure $\sigma$ and $\sigma(\partial\Omega_i\cap \partial\Omega_j) = 0$ whenever $i\ne j$.

With this notation, the infinite-dimensional LP reads
\begin{equation}\label{opt:LPinf}
\begin{array}{ll}
\underset{(\mu,\mu_{\partial_1},\ldots,\mu_{\partial_{n_{\mr{b}}}})}{\mbox{inf}\,/\,\mbox{sup}} & \int_{\Omega\times \Yf\times \Zf}L(x,y,z)\,d\mu + \sum_{i=1}^{n_{\mr{b}}}\int_{\partial\Omega\times \Yf\times \Zf} L_\partial(x,y,z)\,d\mu_{\partial_i}  \\ \vspace{1.1mm}
\mbox{subject to} &  (\ref{eq:thmEq1}),\; (\ref{eq:thmEq2}),\; (\ref{eq:thmEq3}) \;\;[\text{with}\;\, \mu_\partial = \sum_{i=1}^{n_{\mr{b}}}\mu_{\partial_i}] \quad \forall \, \phi \in C^\infty(\Omega\times\Yf)\\ \vspace{1.6mm}
&\int_{\partial\Omega_i\times \Xf\times \Zf} \psi(x)\,d\mu_{\partial_i}(x,y,z) = \int_{\partial\Omega_i} \psi(x)\,\sigma(x)\quad \forall \psi\in C^\infty(\Omega),\;i\in\{1,\ldots,n_{\mr{b}}\}\\ 
& (\mu,\mu_{\partial_1},\ldots,\mu_{\partial_{n_{\mr{b}}}}) \in \Kc,
\end{array}
\end{equation}
where 
 \[
\Kc =  \Mc(\Omega\times \Yf\times \Zf)_+ \times \Mc(\partial\Omega_1\times \Yf\times \Zf)_+\times\ldots\times \Mc(\partial\Omega_{n_{\mr{b}}}\times \Yf\times \Zf)_+.
\]
It is an immediate observation that $\Mc(\cdot)_+$ is a \emph{convex cone},  and therefore $\Kc$ is a also a convex cone. As a result, since all equality constraints of~(\ref{opt:LPinf}) are affine in $(\mu,\mu_{\partial_1},\ldots,\mu_{\partial_{n_{\mr{b}}}})$, the optimization problem~(\ref{opt:LPinf}) is indeed an infinite dimensional LP. The last equality constraint of~(\ref{opt:LPinf}) is a normalization constraint prescribing that the $x$-marginal of $\mu_{\partial_i}$ is equal to the surface measure restricted to $\mu_{\partial_i}$, which trivially holds for any boundary measure defined by~(\ref{eq:boundaryMeas}). Using the integration by parts formula~(\ref{eq:stokes}) and~(\ref{eq:thmEq1}) with $\phi(x,y) = x^\alpha $ for all $\alpha \in \mathbb{N}^n$, this also implies that the $x$-marginal of $\mu$ is equal to the Lebesgue measure on $\Omega$, which holds for any occupation measure defined by~(\ref{eq:occupMeas}).


The following Lemma states that whenever the a pair of measures $(\mu,\mu_\partial)$ feasible in~(\ref{opt:LPinf}) is supported on a graph of a sufficiently regular function $x\mapsto (y(x),z(x))$, then necessarily $\frac{\partial y_k}{\partial x_m} = z_{k,m}(x)$.
\begin{lemma}
Let $(\mu,\mu_\partial)$ be a pair of measures feasible in~(\ref{opt:LPinf}) and assume that $\mu$ and $\mu_{\partial}$ are supported on a graph $\{ (x,y(x),z(x)) \mid x \in\Omega\} \subset \Omega\times \Yf\times \Zf$ for some mappings $y(\cdot) \in W^{1,\infty}$ and $z(\cdot)\in L_\infty$. Then for all $m\in \{1,\ldots, n\}$ and all $k
\in \{1,\ldots, n_y\}$
\begin{equation}\label{eq:lemEq}
\frac{\partial y_k}{\partial x_m}(x) = z_{k,m}(x)
\end{equation}
holds almost everywhere in $\Omega$. In particular, if $y(\cdot) \in C^1(\Omega)$ and $z(\cdot) \in C(\Omega)$, then~(\ref{eq:lemEq}) holds for all $x\in\Omega$.
\end{lemma}
\begin{proof}
Fix $k$ and $m$ and select $\phi(x,y) = \psi(x)y_k$ with $\psi(x) \in C^\infty(\Omega)$. Using the fact that $(\mu,\mu_\partial)$  are feasible in~(\ref{opt:LPinf}) and hence satisfy~(\ref{eq:thmEq1}) and by the assumption that $\mu$ and $\mu_\partial$ are supported on the graphs, we obtain
\[
\int_{\Omega} \Big [y_k(x) \frac{\partial \psi}{\partial x_m} + \psi(x)z_{k,m}(x)\Big]\, dx = \int_{\partial\Omega}\eta_m(x)y_k(x) \psi(x)\,d\sigma(x),
\]
where we disintegrated the measures $\mu$ and $\mu_\partial$ and used the facts that their respective $x$-marginals are equal to the Lebesgue measure on $\Omega$ respectively the surface measure $\sigma$. Applying the integration by parts formula~(\ref{eq:stokes}) to the right hand side of this equation yields
\[
\int_{\Omega} \psi(x) \Big [z_{k,m}(x) -\frac{\partial y_k}{\partial x_m}(x)\Big]\, dx = 0 .
\]
Since $\Omega$ is compact and $\psi$ an arbitrary smooth function, this implies that~(\ref{eq:lemEq}) holds for almost all $x\in\Omega$. If $y_k \in C^1$ and $z_k \in C$, this implies that in fact (\ref{eq:lemEq}) holds for all $x\in\Omega$.
\end{proof}

Now we will prove a result establishing that the LP~(\ref{opt:LPinf}) provides bounds on the optimal value of the nonconvex optimization problem~(\ref{opt:ocp_uncont}) and therefore on the functional~(\ref{eq:obj}) evaluated along the solutions to the PDE~(\ref{eq:pdeEntire}).

\begin{theorem}\label{thm:LPbound}
Let $p_{\mr{inf}}$ and $p_{\mr{sup}}$ denote the optimal values of~(\ref{opt:ocp_uncont}) with infimum respectively supremum and similarly let $\hat p_{\mr{inf}}$ and $\hat p_{\mr{sup}}$ denote the respective optimal values for~(\ref{opt:LPinf}). Then we have
\begin{subequations}\label{eq:bounds}
\begin{align}
\hat p_{\mr{inf}} &\le p_{\mr{inf}},\label{eq:lb} \\ 
\hat p_{\mr{sup}} &\ge p_{\mr{sup}} \label{eq:ub}.
\end{align}
\end{subequations}
\end{theorem}
\begin{proof}
The result follows immediately from  Theorem~\ref{thm:linearRepGen}. Indeed, given any $y(\cdot)$ satisfying the constraints of (\ref{opt:ocp_uncont}) (i.e., solving the PDE (\ref{eq:pdeEntire})), Theorem~\ref{thm:linearRepGen} guarantees that (\ref{eq:thmEq1})-(\ref{eq:thmEq3}) are satisfied by the corresponding occupation and boundary measures $(\mu,\mu_\partial)$ defined by~(\ref{eq:occupMeas}) and (\ref{eq:boundaryMeas}), which are nonnegative measures with supports in $(\Omega\times \Yf\times \Zf)$ respectively $(\partial\Omega\times \Yf\times \Zf)$ and therefore feasible in~(\ref{opt:LPinf}), after performing the decomposition~(\ref{eq:boundMeasDecomp}). The value of the objective functional in~(\ref{opt:LPinf}) associated to this occupation and boundary measures coincides with $J(y(\cdot))$ by virtue of~(\ref{eq:occupMeasFun}) and (\ref{eq:boundaryMeasFun}).  Therefore, the feasible set of~(\ref{opt:ocp_uncont}) is a subset of the feasible set of~(\ref{opt:LPinf}) and the objective functionals coincide on the intersection of the two. As a result, (\ref{eq:lb}) and (\ref{eq:ub}) necessarily hold.
\end{proof}

\begin{remark}\label{rem:tightBounds} It is an interesting and challenging question to investigate  whether the inequalities in~(\ref{eq:lb}) and (\ref{eq:ub}) are strict. In other words, to investigate whether the infinite-dimensional LP~(\ref{opt:LPinf}) is a relaxation or an equivalent reformulation of~(\ref{opt:ocp_uncont}). If (\ref{eq:lb}) and (\ref{eq:ub}) turn out to be relaxations (i.e., the inequalities are strict), one could try to add additional linear or convex constraints to~(\ref{opt:LPinf}) that are necessarily satisfied by the occupation and boundary measures of any solution to~(\ref{eq:pdeEntire}), thereby tightening the constraint set of~(\ref{opt:LPinf}), with the aim to reduce the gap between the optimal values. One type of inequalities, valid for a class of conservation laws, are the so-called entropy inequalities originally introduced in~\cite{diperna1985measure} in a scalar setting and subsequently generalized in, e.g., \cite{brenier2011weak}. The situation, however, appears to be more complex for the general formulation treated in this work, espetially in the controlled setting treated in Section~\ref{sec:statement_cont}.  Let us mention that the absence of such a gap was proven in~\cite{vinter1993convex} for the problem of optimal control of ordinary differential equations and in~\cite{kordaROA, kordaMCI} for the problems of region of attraction and invariant set computation.
\end{remark}

\section{Computable bounds using SDP relaxations}\label{sec:SDPuncont}
In this section we describe how to compute upper and lower bounds on the optimal value of the infinite-dimensional LP~(\ref{opt:LPinf}), which, by Theorem~\ref{thm:LPbound}, will provide bounds on the optimal value of~(\ref{opt:ocp_uncont}). The main tool for doing so is the so-called moment-sums-of-squares or Lasserre hierarchy of semidefinite programming (SDP) relaxations, originally developed for global optimization of polynomials in~\cite{lasserre2001global} and later extended to optimal control of ordinary differential equations~(e.g., \cite{lasserre2008nonlinear,korda2016controller}).

The main step of the approximation procedure is a construction of a sequence of finite-dimensional, SDP representable, cones approximating the cone of nonnegative measures $\Mc(\Kf)_+$ from the \emph{outside}, where $\Kf$ is a basic semialgebraic set of the form
\[
\Kf = \{x \in \Rb^n \mid g_i(x)\ge0, \; i=1,\ldots, n_g\}.
\]
In our case, the set $\Kf$ will be either $\Omega\times \Yf\times \Zf$ or $\partial\Omega_i\times \Yf\times \Zf$, $i=1,\ldots, n_{\mr{b}}$. For now, we shall work with an arbitrary basic semialgebraic set $\Kf$. The finite-dimensional outer approximations $\Mc_d(\Kf)_+$ indexed by the degree $d$ are given by
 \begin{equation}\label{eq:Msup_def}
 \Mc_d(\Kf)_+ = \{ \bs s \in \Rb^{\binom{n+d}{d}} \mid M_d(\bs s) \succeq 0,\;  M_d(g_i, \bs s) \succeq 0, \;i\in\cel{1}{n_g} \},
 \end{equation}
 where $M_d(\bs s)$ and $M_d(g_i, \bs s)$ are the so-called \emph{moment} and \emph{localizing} matrices and $\succeq$ denotes positive semidefiniteness. The moment and localizing matrices are given by
 \begin{subequations}\label{eq:momMatdef}
\begin{align}
M_d(\bs s) &= l_{\bs s}(\beta_{d/2}\beta_{d/2}^\top) \label{eq:momMatOnlyDef} \\
M_d(g_i, \bs s) &= l_{\bs s}(\beta_{\bar d_i}\beta_{\bar d_i}^\top g_i) \label{eq:locMatOnlyDef},
\end{align} 
\end{subequations}
where $\bar d_i = \lfloor (d - \mr{deg}\,g_i )/ 2\rfloor$ and $\beta_d$ is the basis vector of all monomials of total degree at most $d$ with a given ordering, choice of which is immaterial. For example, with the graded lexicographic ordering and $n = 2$, $d = 3$, one has
\[
\beta_d =  [1,\,x_1,\,x_2,\,x_1^2,\,x_1x_2,\,x_2^2,\,x_1^3,\,x_1^2x_2,\,x_1x_2^2,\,x_2^3]^\top.
\]
The symbol $l_{\bs s}(\cdot)$ denotes the Riesz functional defined for any polynomial\footnote{In~(\ref{eq:momMatdef}), the Riesz functional $\ell_{\bs s}(\cdot)$ is applied elementwise to the matrix polynomials $\beta_{d/2}(x) \beta_{d/2}(x)^\top$ and $\beta_{\bar d_i}(x) \beta_{\bar d_i}(x)^\top$.} of the form \[
p(x) = \sum_{|\alpha| \le d} c_\alpha x^\alpha\]
 by
\begin{equation}\label{eq:riesz}
l_{\bs s}(p) = \sum_{|\alpha| \le d}  c_\alpha  s_\alpha = \bs c^\top \bs s,
\end{equation}
where we are using a  multi-index notation, i.e., $x^\alpha = x_1^{\alpha_1} \cdot\ldots \cdot x_n^{\alpha_n}$, $|\alpha| = \sum_{i=1}^n\alpha_i$ and the coefficients $\bs c = ( c_\alpha)_{|\alpha| \le d}\in \Rb^{\binom{n+d}{d}}$ and the truncated moment vector $\bs s = (s_\alpha)_{|\alpha| \le d} \in \Rb^{\binom{n+d}{d}}$ have the same ordering. Notice in particular that if $\bs s$ is in fact a truncated moment vector of a given measure $\mu \in \Mc(\Kf)_+$, i.e.,
\[
s_\alpha = \int_{\Kf} x^\alpha\,d\mu,
\]
then
\[
  \int_{\Kf} p\,d\mu = \sum_{|\alpha| \le  d} c_\alpha \int_{\Kf} x^\alpha\, d\mu = l_{\bs s}(p).
\]
The Riesz functional is, however, defined for any vector $\bs s$, not necessarily equal to a truncated moment vector of a nonnegative measure. We note that the cone $M_d(\Kf)$ is a spectrahedron (i.e., a feasible set of a linear matrix inequality) and hence linear functionals can be minimized or maximized over this set using convex SDPs. This follows immediately from the fact that the matrices $M_d(\bs s)$ and $M_d(g_i, \bs s)$ are symmetric and depend linearly on~$\bs s$.

Now we apply these general considerations to our setting. For each $d\in\Nb$, we can define the SDP representable cones
\[
\Mc_d(\Omega\times \Yf\times \Zf)_+
\]
and
\[
\Mc_d(\partial\Omega_i \times \Yf\times \Zf)_+, \;\; i=1,\ldots, n_{\mr{b}},
\]
where $\partial\Omega_i$ are the basic semialgebraic elements of the partition~(\ref{eq:boundPartiotion}) of the boundary of $\Omega$.

The last step in the approximation procedure is to restrict the test functions in the linear equality constraints of~(\ref{opt:LPinf}) (i.e., in~(\ref{eq:thmEq1})-(\ref{eq:thmEq3})) to monomials of total degree no more than $d'$ (to be specified later), i.e., to pick $\phi(x,y) = x^\alpha y^\gamma$ with $|(\alpha,\gamma)| \le d'$. Assuming that the unit surface normal vector $\eta(x)$ is polynomial on each element $\partial\Omega_i$ of the partition of the boundary, we observe that all integrands in~(\ref{eq:thmEq1})-(\ref{eq:thmEq3}) are polynomials in $(x,y,z)$. The degree $d'$ is selected to be the largest possible degree such that the degree of all the polynomials appearing under the integral sign in~(\ref{eq:thmEq1})-(\ref{eq:thmEq3}) is no more than $d$ (the exact value of $d'$ depends on the degrees of $F$, $B_{i,j}$ and $G$). The test functions $\psi$ are chosen as $\psi(x) = x^\alpha$, $|\alpha|\le d$. With this choice of test functions, the constraints of (\ref{opt:LPinf}) can be rewritten as the finite-dimensional linear equation
\[
A_d \bs s = b_d
\]
for some matrix and $A_d$ and vector $b_d$, where
\[
\bs s := \begin{bmatrix}\bs s_\mu \\ \bs s_{\partial_1} \\ \vdots \\ \bs s_{\partial_{n_{\mr{b}}}}  \end{bmatrix}
\]
with $\bs s_\mu$ being the truncated vector of moments of $\mu$ up to the total degree~$d$ and analogously $\bs s_{\partial_i}$ being the truncated vector of moments of $\mu_{\partial_i}$ from the decomposition~(\ref{eq:boundMeasDecomp}). Similarly, assuming the degree of the polynomials $L$ and $L_\partial$ is at most~$d$, the objective functional of~(\ref{opt:LPinf}) can be rewritten as
\[
c_d^\top \bs s
\]
for some vector $c_d$. The degree $d$ finite-dimensional SDP relaxation of~(\ref{opt:LPinf}) then reads
\begin{equation}\label{opt:SDP}
\begin{array}{ll}
\underset{\bs s}{\mbox{inf}\,/\,\mbox{sup}} & c_d^\top \bs s  \\
& A_d\bs s = b_d\\
& \bs s \in \Kc_d,
\end{array}
\end{equation}
where
\[
\Kc_d := \Mc_d(\Omega\times \Yf\times \Zf)_+ \times \Mc_d(\partial\Omega_1\times \Yf\times \Zf)_+\times \ldots\times \Mc_d(\partial\Omega_{n_{\mr{b}}}\times \Yf\times \Zf)_+
\]
is a convex SDP representable cone and therefore~(\ref{opt:SDP}) is indeed an SDP problem. The entries of the matrix $A_d
$ and the vectors $b_d$ and $c_d$ depend on the coefficients of the polynomials $F$, $B_{i,j}$ and the polynomials comprising the piecewise polynomial function $G$. With the exception of $c_d$, these are rather cumbersome expressions and we omit them for brevity. Fortunately, the entire SDP~(\ref{opt:SDP}) can be assembled automatically using existing software for generalized moment problems such as Gloptipoly~3, \cite{gloptipoly3}, with the only input being $F$, $B_{i,j}$, $G$ and the relaxation degree $d$. The SDP can then be solved by any of the available SDP solvers such as MOSEK or SeDuMi~\cite{sedumi}.

The optimal values of the SDP~(\ref{opt:SDP}) provide bounds on the optimal values of the infinite dimensional LP~(\ref{opt:LPinf}) and hence also on the optimal values of the original problem~(\ref{opt:ocp_uncont}), via Theorem~\ref{thm:LPbound}; in addition, assuming the sets $\Omega$, $\Yf$ and $\Zf$ are compact, these SDP bounds  converge to the bounds provided by the infinite dimensional LP~(\ref{opt:LPinf}):
\begin{theorem}\label{thm:conv}
Let $\hat{p}_{\mr{inf},d}^{\mr{SDP}}$ resp. $\hat{p}_{\mr{sup},d}^{\mr{SDP}}$ denote the optimal values of the degree $d$ SDP relaxations~(\ref{opt:SDP}) with infimum resp. supremum.  Then we have
\begin{subequations}
\begin{align}
\hat{p}_{\mr{inf},d}^{\mr{SDP}} \le \hat{p}_{\mr{inf},{d+1}}^{\mr{SDP}} \le \hat{p}_{\mr{inf}} \le p_{\mr{inf}} \label{eq:SDPlb}\\
\hat{p}_{\mr{sup},d}^{\mr{SDP}} \ge \hat{p}_{\mr{sup},d+1}^{\mr{SDP}} \ge \hat{p}_{\mr{sup}} \ge p_{\mr{sup}}\label{eq:SDPub}
\end{align}
\end{subequations}
for all $d\in\Nb$. In addition
\begin{subequations}
\begin{align}
\lim_{d\to\infty}\hat{p}_{\mr{inf},d}^{\mr{SDP}} = \hat{p}_{\mr{inf}}  \label{eq:SDPlbConv}\\
\lim_{d\to\infty}\hat{p}_{\mr{sup},d}^{\mr{SDP}} = \hat{p}_{\mr{sup}}  \label{eq:SDPubConv}
\end{align}
\end{subequations}
provided that $\Omega$, $\Yf$ and $\Zf$ are compact and the polynomials defining these sets satisfy the Archimedianity condition\footnote{The Archimedianity condition (e.g., \cite[Definition 3.18]{laurent2009sums}) is a nonrestrictive algebraic condition implying compactness. This condition can always be satisfied by adding redundant ball constraints of the form $N^2 -x^\top x \ge 0$ to the description of the \emph{compact} basic semialgebraic sets $\Omega$, $\Yf$, $\Zf$ for a sufficiently large $N$. Adding such ball constraints also ensures that the infimum or supremum in~(\ref{opt:SDP}) is always attained. However, Theorem~\ref{thm:conv} is valid whenever the Archimedianity condition holds, irrespective of whether redundant ball constraints are present.\label{foot:arch}}.
\end{theorem}

\begin{proof}
The proof is a standard argument, following e.g., \cite{lasserre2001global}. The first two inequalities in~(\ref{eq:SDPlb}) and (\ref{eq:SDPub}) follow by construction since the constraint set of the SDP~(\ref{opt:SDP}) is a relaxation of the constraint set of the LP~(\ref{opt:LPinf}) and since the constraint set of (\ref{opt:SDP}) tightens with increasing $d$. These facts follow immediately from the definition of the truncated moment cone~(\ref{eq:Msup_def}) using the moment and localizing matrices~(\ref{eq:momMatdef}) and from the fact that more equality constraints are added to~(\ref{opt:SDP}) with increasing $d$, corresponding to imposing the constraints of~(\ref{opt:LPinf}) for increasing degrees of the test functions $\phi$ and $\psi$. The key step in proving convergence of the lower and upper bounds consists in establishing boundedness of the truncated moment sequences optimal in~(\ref{opt:SDP}) (in fact the argument below establishes boundedness of any feasible truncated moment sequence). Let therefore $\bs s^d$ denote a truncated moment sequence optimal\footnote{For simplicity of the argument we assume that the infimum or supremum is attained in~(\ref{opt:SDP}); if it is not, the result follows by considering minimizing sequences. See also Footnote~\ref{foot:arch}.} in~(\ref{opt:SDP}) with components $\bs s^d_\mu(0)$ and $\bs s^d_{\partial_i}(0)$, $i=1,\ldots,n_{\mr{b}}$. The main step is to establish boundedness of the  first components, i.e., of $\bs s^d_\mu(0)$ and $\bs s^d_{\partial_i}(0)$, $i=1,\ldots,n_{\mr{b}}$. For $\bs s^d_{\partial_i}(0)$ this follows readily by considering the  test function $\psi = 1$ for the last constraint of~(\ref{opt:LPinf}), which implies $\bs s^d_{\partial_i}(0)  = \int_{\partial\Omega_i}1d\sigma <\infty$. For $\bs s^d_{\mu}(0)$ we consider $\phi = x_m$ in~(\ref{eq:thmEq1}) for any $m$ such that  $\eta_m \ne 0$, obtaining $\int 1d\mu = \int_{\partial\Omega_m} x_m \eta_m(x)\,d\sigma(x) < \infty$, which for $d \ge \mr{deg}(\eta_m) + 1$ implies that $\bs s^d_\mu(0) = \int_{\partial\Omega_m} x_m \eta_m(x)\,d\sigma(x) < \infty$ by virtue of the last equality constraint of~(\ref{opt:LPinf}) evaluated for all $\psi = x^\alpha$, $|\alpha|\le d$. Having proved that $\bs s^d_\mu(0)$ and $\bs s^d_{\partial_i}(0)$, $i=1,\ldots,n_{\mr{b}}$, are bounded, the componentwise boundedness of the whole moment sequences as $d$ tends to infinity follows from the Archimedianity condition and the structure of the moment and localizing matrices. Therefore, using a standard diagonal argument, there exists a subsequence $(\bs s^{d_i})$ converging elementwise to some sequence $\bs s^\star$. By continuity of the minimum eigenvalue of a matrix it holds that $\bs s^\star_\mu \in \Mc_d(\Omega\times \Yf\times \Zf)_+$ and $\bs s^\star_{\partial_i} \in \Mc_d(\partial\Omega_i\times \Yf\times \Zf)_+$ for all $i=1,\ldots,n_{\mr{b}}$ and all $d\in \Nb$. This, the Archimedianity condition and the Putinar positivestellensatz ensure that the sequences $\bs s^\star_\mu$ and $\bs s^\star_{\partial_i}$
 admit nonnegative representing Borel measures $\mu^\star$, $\mu_{\partial_i}^\star$. For concreteness, consider infimum in~(\ref{opt:LPinf}), the argument being the same for supremum. By continuity, the measures $\mu^\star$, $\mu^\star_{\partial_i}$ satisfy the constraints of~(\ref{opt:LPinf}) and hence the objective value attained by them is greater than or equal to $\hat{p}_{\mr{inf}}$. On the other hand, by continuity of the objective functional, the value attained by these measures is equal to $\lim_{d\to\infty}\hat{p}_{\mr{inf},d}^{\mr{SDP}}$. But $\hat{p}_{\mr{inf},d}^{\mr{SDP}} \le  \hat{p}_{\mr{inf}}$ for all $d$ and hence necessarily $\lim_{d\to\infty}\hat{p}_{\mr{inf},d}^{\mr{SDP}} = \hat{p}_{\mr{inf}}$.
\end{proof}

\subsection{Non-polynomial boundary conditions}\label{sec:nonPolBnd}
In this section we treat the case of Dirichlet boundary conditions of the form $y(x) = h_i(x)$, $x\in\partial\Omega_i$, $i \in \mathcal{I}$, $\mathcal{I}\subset \{1,\ldots,n_{\mr{b}}\}$, with $h_i$ being an arbitrary integrable function. We note that the approach presented in this section is preferable also if $h_i$ is polynomial, provided that one wishes to impose the simple constraints $y(x) = h_i(x)$ rather than the more general constraint~(\ref{eq:bnd}). This type of boundary condition translates immediately to the condition
\begin{equation}\label{eq:bndNonPol}
\int_{\partial\Omega_i\times\Yf\times\Zf} \phi(x,y) \, d\mu_{\partial_i}(x,y,z) = \int_{\partial\Omega_i} \phi(x,h_i(x))\, d\sigma(x),
\end{equation}
where we used~(\ref{eq:boundaryMeasFun}) with $g(x,y,z) = \phi(x,y)$, substituted the constraint $y(x) = h_i(x)$ and used the decomposition~(\ref{eq:boundMeasDecomp}). We note that, given $\phi$, the right hand side of~(\ref{eq:bndNonPol}) is constant and hence~(\ref{eq:bndNonPol}) is affine in $\mu_{\partial_i}$. This equation can then be used instead of~(\ref{eq:thmEq3}) in the infinite dimensional LP~(\ref{opt:LPinf}). As far as the finite-dimensional SDP relaxation goes, selecting $\phi$ to be the vector of all monomials in~$(x,y)$ up to degree $d$  immediately leads to a system of  linear equations in the truncated moment sequence $\bs s_{\partial_i}$; the right hand side of this equation can be precomputed provided that the integrals $\int_{\partial\Omega_i} x^\alpha h_i^\beta(x),d\sigma(x)$, $|(\alpha,\beta)|\le d$, can be evaluated (either analytically or numerically). This set of linear equations is then used as a part of the equality constraints of SDP~(\ref{opt:SDP}). We note that the constraint~(\ref{eq:bndNonPol}) can be combined with~(\ref{eq:thmEq3}) modeling other boundary conditions of the generic type~(\ref{eq:bnd}).

\subsection{Periodic boundary conditions}\label{sec:periodicBnd}
Assume that there exists a polynomial mapping $h:\partial\Omega_i\to \partial\Omega_j$, $i\ne j$, preserving the surface measure, i.e., $\sigma(h^{-1}(A)) = \sigma(A)$ for all Borel $A \subset \Omega_{\partial_j}$ (in other words, $\sigma_{|\partial\Omega_j}$ is the push-forward by $h$ of $\sigma_{|\partial\Omega_i}$). The goal is to impose the constraint
\begin{equation}\label{eq:periodConst}
y(x) = y(h(x))\;\; \forall x\in \partial\Omega_i.
\end{equation}
For any test function $\phi(x,y)$ we have
\[
\int_{\partial\Omega_j} \phi(x,y(x))\,d\sigma(x) = \int_{\partial\Omega_i} \phi(h(x),y(h(x)))\,d\sigma = \int_{\partial\Omega_i} \phi(h(x),y(x))\,d\sigma,
\]
where we used the measure preserving property of $\sigma$ in the first equality and the constraint~(\ref{eq:periodConst}) in the second. Using~(\ref{eq:boundaryMeasFun}) and~(\ref{eq:boundMeasDecomp}), we get
\begin{equation}\label{eq:periodicBnd}
\int_{\partial\Omega_j\times\Yf\times\Zf} \phi(x,y)\,d\mu_{\partial_j}(x,y,z)  = \int_{\partial\Omega_i\times\Yf\times\Zf} \phi(h(x),y)\,d\mu_{\partial_i}(x,y,z),
\end{equation}
which is a linear equation in $(\mu_{\partial_i},\mu_{\partial_j})$. Selecting $\phi(x,y)$ to be the vector of all monomials up to degree $d - \mr{deg}(h)$ leads to a set of equality constraints among the  moment sequences $(\bs s_{\partial_i},\bs s_{\partial_j})$ of the measures $(\mu_{\partial_i},\mu_{\partial_j})$ truncated up to degree $d$. These equality constraints are then added to the constraints of the SDP~(\ref{opt:SDP}).

\paragraph{Example} {\it As an example, consider the rectangular domain $\Omega =  [0,1] \times [0,a]$, where we wish to make the second variable periodic. This can be used to model a PDE with one time variable and one spatial variable on a circular spatial domain. This boundary condition is imposed by~(\ref{eq:periodicBnd}) with $h(x) = (x_1,x_2 + a)$.}

\subsection{Surface normal vector}\label{sec:normalVec}
In this section we describe how the assumption of the unit surface vector $\eta$ being polynomial can be relaxed. Since the set $\Omega$ is basic semialgebraic of the form~(\ref{eq:omegaDef}), for each element $\Omega_{\partial,i}$ of the boundary partition~(\ref{eq:boundPartiotion}) there exists\footnote{The polynomials $h_i$ can be determined from the polynomials $g_j^\Omega$ defining $\Omega$ in~(\ref{eq:omegaDef}). Indeed, each $h_i$ is equal to the $g_j^\Omega$  for which $g_j^\Omega = 0$ on $\Omega_{\partial,i}$, provided that $\nabla g_j^\Omega \ne 0$ on $\Omega_{\partial,i}$ (otherwise it may be equal to a factor of $g_j^\Omega$).} a polynomial $h_i$ such that
\[
\eta(x) = \frac{\nabla h_i(x)}{\| \nabla h_i(x)\|},\quad x\in\Omega_{\partial,i}
\]
with $\nabla h_i \ne 0$ on $\Omega_{\partial,i}$. Now we can define the measures
\[
\sigma'_i = \frac{\sigma_i}{\|  \nabla h_i\|},
\]
where $\sigma_i$ is the restriction of the surface measure $\sigma$ to $\Omega_{\partial,i}$. With this notation, the integration by parts formula~(\ref{eq:stokes}) becomes
\begin{equation}\label{eq:stokesMod}
\sum_{i=1}^{n_{\mr{b}}}\int_{\partial\Omega} g(x) [\nabla h_i]_m(x)\,d\sigma'_i(x) = \int_{\Omega} \frac{\partial g}{\partial x_m}\,dx,\;\; m=1,\ldots,n,
\end{equation}
where $ [\nabla h_i]_m$ denotes the $m^{\mr{th}}$ component of $\nabla h_i$. In view of the definition of $\mu_\partial$~(\ref{eq:boundaryMeas}), it immediately follows that the equations~(\ref{eq:thmEq1})-(\ref{eq:thmEq3}) hold with $\eta_m$ replaced\footnote{To be precise, any integral of the form $\int g\, \eta_m\, d\mu_\partial$ is replaced by $\sum_{i=1}^{n_{\mr{b}}}\int g\, [\nabla h_i]_m  d\mu_{\partial_i}'$.}  by $[\nabla h_i]_m$ and $\mu_\partial$ replaced by
\[
\sum_{i=1}^{n_{\mr{b}}}\mu_{\partial_i}',
\]
where \[
\mu_{\partial_i}' = \frac{\mu_{\partial_i}}{\|\nabla h_i\|}.
\]
Crucially, since $h$ is polynomial, so is $\nabla h$ and therefore the integrands in (\ref{eq:thmEq1})-(\ref{eq:thmEq3}) are all polynomials as is required by the subsequent SDP relaxation procedure. The only modification to the infinite-dimensional LP (\ref{opt:LPinf}), which is the starting point for the relaxation, is the replacement of $\eta_m$ by $[\nabla h_i]_m$ in the constraints (\ref{eq:thmEq1})-(\ref{eq:thmEq3}) and the replacement of $\sigma$ by $\sigma'$ in the last equality constraints of this LP. The decision variables are re-labeled to $(\mu, \mu_{\partial_1}',\ldots,\mu_{\partial_{n_{\mr{b}}}}')$. If so desired, the boundary part of the original solution can be recovered as $\mu_{\partial_i} = \|\nabla h_i\| \mu_{\partial_i}'$.

The same considerations hold for the PDE control part of this paper treated in the subsequent sections.

\section{Problem statement (control)}\label{sec:statement_cont} 
In this section, we describe the problem setting for control of nonlinear PDEs. We consider a controlled nonlinear PDE of the form
\begin{subequations}\label{eq:pdeEntireCont}
\begin{align}
F(x,y(x),\Dc y(x)) + \sum_{i,j} B_{i,j}(x,y(x)) \frac{\partial ^2 y}{\partial x_i\partial x_j} & = C(x,y(x))u(x),\quad \hspace{5mm} x\in \Omega, \label{eq:pdeCont} \\ \label{eq:bndCont}
\forall i\in\{1,\ldots, n_{\mr{b}}\}\;\;G_i(x,y(x),\Dc y(x)) &= C_{\partial_i}(x,y(x)) u_{\partial_i}(x),\quad x \in \partial\Omega_i,
\end{align}
\end{subequations}
where is $u(\cdot)$ the distributed control and $u_{\partial_i}(\cdot)$ are the boundary controls defined on each element of the partition of the boundary~(\ref{eq:boundPartiotion}). The functions $C:\Omega\to \Rb^{n_F\times n_u}$ and $C_{\partial_i}:\partial\Omega_i\to \Rb^{n_{G_i}\times n_{u_i}}$ are given matrix functions with each entry being a multivariate polynomial in~$(x,y)$. Similarly, the vector functions $G_i:\Omega\times \Rb^{n_y}\times \Rb^{n_y\times n}\to\Rb^{n_{G_i}}$ are assumed to be multivariate polynomials in all variables. Otherwise, the setup is the same as in Section~\ref{sec:statement_anal}.

 The control goal is to minimize the functional
\begin{align}\label{eq:objCont}
J\big(y(\cdot), u(\cdot)\big) := &\int_{\Omega}L\left(x,y(x),\Dc y(x)\right)\,dx + \sum_{i=1}^{n_{\mr{b}}}\int_{\Omega_{\partial,i}}L_{\partial_i}\left(x,y(x),\Dc y(x)\right)\,d\sigma(x), \\
&+\int_{\Omega}L_u\left(x,y(x)\right)^\top u(x)\,dx + \sum_{i=1}^{n_{\mr{b}}}\int_{\Omega_{\partial,i}}L_{u_i}\left(x,y(x)\right)^\top u_{\partial_i}(x)\,d\sigma(x), 
\end{align}
subject to constraints on $y$, $\Dc y$ and the control inputs. Here $L$ and $L_{\partial_i}$ are polynomials and $L_u$ and $L_{u_i}$ are vectors of polynomials. This  leads to the following optimal control problem

\begin{equation}\label{opt:ocp_cont}
\begin{array}{ll}
\underset{y(\cdot),\, u(\cdot)}{\mbox{inf}\,/\,\mbox{sup}} & J\big(y(\cdot), u(\cdot)\big)  \\
\mbox{subject to} &  (\ref{eq:pdeCont}),\; (\ref{eq:bndCont})\\
& y(x) \in \Yf\; \forall x\in \Omega,\quad y(x)\in \Yf_{\partial_i}\;\;\;\, \forall x\in \partial\Omega_i,\;\; i\in \{1,\ldots,n_{\mr{b}}\} \\
& \Dc y(x) \in \Zf\; \forall x\in \Omega,\; \Dc y(x)\in \Zf_{\partial_i}\; \,\forall x\in \partial\Omega_i,\;\; i\in \{1,\ldots,n_{\mr{b}}\} \\
&u(x) \in \Uf\; \forall x\in\Omega, \quad u_{\partial_i}(x)\in \Uf_{\partial_i}\; \forall x\in \partial\Omega_i,\;\; i\in \{1,\ldots,n_{\mr{b}}\}
\end{array}
\end{equation}
where $\Yf$ and $\Zf$ as well as $\Yf_{\partial_i}$ and $\Zf_{\partial_i}$, $i=1,\ldots, n_{\mr{b}}$, are basic semialgebraic sets. The sets $\Uf$ and $\Uf_{\partial_i}$ are assumed to be the unit boxes
\begin{equation}\label{eq:boxInput}
\Uf = [0,1]^{n_u},\quad \Uf_{\partial_i} = [0,1]^{n_{u_i}}\,.
\end{equation}
We note that any box constraints of the from $[-u^{\mr{min}}_1,u^{\mr{max}}_1]\times\ldots\times [-u^{\mr{min}}_{n_u},u^{\mr{max}}_{n_u}]$ can be transformed to the unit boxes by a simple affine transformation of the functions $C$, $F$ resp. $C_{\partial_i}$, $G_i$.  The input constraints~(\ref{eq:boxInput}) therefore cover all box constraints which are the most commonly encountered type of constraint in practice.


\section{Linear representation (control)}\label{sec:linRepCont}
In this section we derive an infinite-dimensional linear program in the space of Borel measures whose solution will provide a lower bound on~(\ref{opt:ocp_cont}); we also show how the solution to this LP can be used to extract a controller in a feedback form. This infinite-dimension LP is then approximated in Section~\ref{sec:SDPcont} via finite-dimensional SDP relaxations, the solutions of which provide controllers for the PDE~(\ref{eq:pdeEntireCont}) in a feedback form. The approach closely follows the method of~\cite{korda2014controller,majumdar2014convex} developed for controlled ordinary differential equations.

As in Section~\ref{sec:linRep} we consider the occupation and boundary measures $\mu$ and $\mu_\partial$ defind by~(\ref{eq:occupMeas}) and (\ref{eq:boundaryMeas}), with the decomposition of the boundary measure~(\ref{eq:boundMeasDecomp}). In addition, we define the vectors of \emph{control measures} $\nu \in \Mc(\Omega\times \Yf)_+^{n_u}$ and $\nu_i \in \Mc(\partial\Omega_i\times \Yf)_+^{n_{u_i}}$
by
\begin{subequations}
\begin{align}
d\nu &= u\, d\bar\mu\label{eq:nuDef}\\
d\nu_{\partial_i} &= u_{\partial_i} d\bar\mu_{\partial_i},\label{eq:nuiDef}
\end{align}
\end{subequations}
where $\bar \mu$ respectively $\bar \mu_{\partial_i}$ are the $(x,y)$ marginals\footnote{The $(x,y)$ marginal of a measure $\mu\in \Mc(\Omega\times \Yf\times \Zf)_+$ is defined by $\bar{\mu}(A\times B) = \mu(A\times B\times \Zf)$ for all Borel sets $A\subset \Omega$, $B\subset \Yf$.} of $\mu$ respectively $\mu_{\partial_i}$.
In other words, we have $\nu \ll \bar\mu$ and $\nu_i \ll \bar\mu_{\partial_i}$ with  density (i.e, the Radon-Nikodym derivative) equal to $u(\cdot)$, respectively $u_{\partial_i}(
 \cdot)$. Here $\nu\ll \bar\mu$ signifies that the measure $\nu$ is absolutely continuous w.r.t. the measure $\bar\mu$. Therefore, for any bounded Borel measurable function $h:\Omega\times \Yf \to \Rb$ we have
\begin{subequations}
\begin{equation}
\int_{\Omega} h(x,y(x))\,u(x)\,dx = \int_{\Omega\times\Yf} h(x,y)\,d\nu(x,y),
\end{equation}
\begin{equation}
\int_{\partial\Omega_i} h(x,y(x))\,u_{\partial_i}(x)\,\sigma(x) = \int_{\partial\Omega_i\times\Yf} h(x,y)\,d\nu_{\partial_i}(x,y).
\end{equation}
\end{subequations}

The following theorem is an immediate generalization of~Theorem~\ref{thm:linearRepGen}.

\begin{theorem}\label{thm:linearRepCont}
Let $u(\cdot) \in L_\infty(\Omega,dx)$, $u_{\partial_i}(\cdot)\in L_\infty(\Omega,d\sigma)$, $i=1,\ldots,n_{\mr{b}}$, and let  $y(\cdot) \in W^{k,\infty}(\interior\Omega;{\bf Y})\cap C(\Omega;{\bf Y}) $ ($k=1$ if $B_{i,j} = 0$, $k=2$ otherwise) satisfy~(\ref{eq:pdeCont}) and (\ref{eq:bndCont}) almost everywhere with respect to the Lebesgue measure on~$\interior\Omega$ respectively the surface measure $\sigma$ on~$\partial\Omega$. If  $\mu$ and $\mu_\partial$ are defined by~(\ref{eq:occupMeas}) and (\ref{eq:boundaryMeas}) and $\nu$ and $\nu_{\partial_i}$ by~(\ref{eq:nuDef}) and (\ref{eq:nuiDef}), then for all $\phi \in C^\infty(\Omega\times \Yf)$
\begin{subequations}\label{eq:thmEqEntireCont}
\begin{align}\label{eq:thmEq1Cont}
\int_{\Omega\times \Yf\times \Zf} \left[ \frac{\partial \phi}{\partial x_m} + \sum_{k=1}^{n_y} \frac{\partial \phi}{\partial y_k}z_{k,m}\right]d\mu(x,y,z) -  \int_{\partial\Omega\times \Yf\times \Zf} \phi(x,y) \eta_m(x)\,d\mu_\partial(x,y,z) = 0
\end{align}
for all $m\in \{1,\ldots,n\}$ and
\begin{align}\label{eq:thmEq2Cont}
\int_{\Omega\times \Yf\times \Zf} \phi(x,y)F(x,y,z)\,d\mu(x,y,z) &\nonumber\\
&\hspace{-60mm} - \sum_{i,j}\int_{\Omega\times \Yf\times \Zf}  \phi(x,y)\left[\frac{\partial B_{i,j}}{\partial x_j}+ \sum_{k=1}^{n_y}\frac{\partial B_{i,j}}{\partial y_k} z_{k,j} \right] z_{\bullet,i}\,d\mu(x,y,z)\nonumber \\ & \hspace{-60mm}  - \sum_{i,j}\int_{\Omega\times\Xf\times\Yf}   \left[\frac{\partial \phi}{\partial x_j} +\sum_{k=1}^{n_y} \frac{\partial \phi}{\partial y_k}z_{k,j}  \right]B_{i,j}(x,y) z_{\bullet,i}\,d\mu(x,y,z)\nonumber\\
& \hspace{-60mm} + \sum_{i,j}\int_{\partial\Omega\times\Yf\times\Zf} B_{i,j}(x,y)\phi(x,y)z_{\bullet,i}\eta_j(x)\,d\mu_\partial(x,y,z) = \sum_{k=1}^{n_u}\int_{\Omega\times \Yf} C_{\bullet,k}(x,y)\phi(x,y) d\nu^k(x,y)
\end{align}
and
\begin{equation}\label{eq:thmEq3Cont}
\int_{\partial\Omega\times \Xf \times \Zf} \phi(x,y) G(x,y,z)\,d\mu_\partial(x,y,z) = \sum_{i=1}^{n_\mr{b}}\sum_{k=1}^{n_{u_i}}\int_{\partial \Omega_i\times \Yf} C_{\bullet,k}^i(x,y)\phi(x,y) d\nu_{\partial_i}^k(x,y),
\end{equation}
where $z_{\bullet,i}$ denotes the $i^{\mr{th}}$ column of the matrix variable $z \in \Rb^{n_y\times n}$ and $C_{\bullet,k}$ respectively $C_{\bullet,k}^i$ the $k^{\mr{th}}$ columns of $C$ respectively $C_{\partial_i}$. The symbols $\nu^k$ and $\nu_{\partial_i}^k$ denote the $k^{\mr{th}}$ components of the vector measures $\nu$ and $\nu_{\partial_i}$. 
\end{subequations}
\end{theorem}
\begin{proof}
The proof is analogous to the proof of Theorem~\ref{thm:linearRepGen}.
\end{proof}

Note that the equations~(\ref{eq:thmEq1Cont})-(\ref{eq:thmEq3Cont}) are linear in the variables $(\mu,\mu_\partial,\nu,\nu_{\partial_1},\ldots,\nu_{\partial_{n_{\mr{b}}}})$. This will allow us to write down an infinite-dimensional LP relaxation of the optimal control problem~(\ref{opt:ocp_cont}).

Before doing so we describe how to handle the requirement that $\nu \ll \bar \mu$ and $\nu_{\partial_i} \ll \bar \mu_{\partial_i}$ with density (= the control input) in $[0,1]$. This is equivalent to requiring that $\nu \le \bar \mu$ and $\nu_{\partial_i} \le \bar \mu_{\partial_i}$, where $\le$ is the ordinary inequality sign (i.e. $\nu(A) \le \mu(A)$ for all Borel sets $A$). This in turn is equivalent to the existence of nonnegative slack measures $\hat \nu \in \Mc(\Omega\times\Yf)_+$, $\hat\nu_{\partial_i}\in \Mc(\partial\Omega\times\Yf)_+$, $i=1,\ldots,n_{\mr{b}}$, such that
\begin{align*}
\nu + \hat \nu &= \bar{\mu} \\
\nu_{\partial_i} + \hat \nu_{\partial_i} &= \bar{\mu}_{\partial_i}, \;i\in\{1,\ldots,n_{\mr{b}}\}.
\end{align*}
Finally, these equalities are equivalent to
\begin{subequations}
\begin{align}
\int_{\Omega\times \Yf}\phi(x,y)\,d\nu(x,y) + \int_{\Omega\times \Yf}\phi(x,y)\,d\hat\nu(x,y) &= \int_{\Omega\times \Yf\times\Zf}\phi(x,y)\, d\mu(x,y,z) \label{eq:hat1} \\
\int\limits_{\partial\Omega_i\times \Yf}\phi(x,y)\,d\nu_{\partial_i}(x,y) + \int\limits_{\partial\Omega_i\times \Yf}\phi(x,y)\, d\hat \nu_{\partial_i}(x,y) &= \hspace{-3mm}\int\limits_{\partial\Omega_i\times \Yf\times \Zf} \hspace{-3mm}\phi(x,y) \, d\mu_{\partial_i}(x,y,z), \; i\in\{1,\ldots,n_{\mr{b}}\} \label{eq:hat2}
\end{align}
\end{subequations}
holding for all test functions $\phi\in C^\infty(\Omega\times\Yf).$
\subsection{Infinite dimensional LP (control)}
Now we are ready to state the infinite-dimensional linear programming relaxation of~(\ref{opt:ocp_cont}). The relaxation reads
\begin{equation}\label{opt:LPinfCont}
\begin{array}{ll}
\underset{\substack{ \mu,\mu_{\partial_1},\ldots,\mu_{\partial_{n_{\mr{b}}}}, \\ \nu,\nu_{\partial_1},\ldots,\nu_{\partial_{n_{\mr{b}}}}, \\ \hat\nu,\hat\nu_{\partial_1},\ldots,\hat\nu_{\partial_{n_{\mr{b}}}} } }{\mbox{inf}} & J_{\mr{LP}}(\mu,\mu_1,\ldots,\mu_{n_{\mr{b}}},\nu,\nu_1,\ldots,\nu_{n_{\mr{b}}})  \\
\mbox{subject to} &  (\ref{eq:thmEq1Cont}),\; (\ref{eq:thmEq2Cont}),\; (\ref{eq:thmEq3Cont}),\;(\ref{eq:hat1}),\;(\ref{eq:hat2}) \;\;[\text{with}\;\, \mu_\partial = \sum_{i=1}^{n_{\mr{b}}}\mu_{\partial_i}] \quad \forall \, \phi \in C^\infty(\Omega\times\Yf) \vspace{1mm}\\
&\int_{\partial\Omega_i\times \Xf\times \Zf} \psi(x)\,d\mu_{\partial_i}(x,y,z) = \int_{\partial\Omega_i} \psi(x)\,\sigma(x)\quad \forall \psi\in C^\infty(\Omega),\;i\in\{1,\ldots,n_{\mr{b}}\}\vspace{1mm}\\
& (\mu,\mu_{\partial_1},\ldots,\mu_{\partial_{n_{\mr{b}}}},  \nu,\nu_{\partial_1},\ldots,\nu_{\partial_{n_{\mr{b}}}},  \hat\nu,\hat\nu_{\partial_1},\ldots,\hat\nu_{\partial_{n_{\mr{b}}}})  \in \Kc
\end{array}
\end{equation}
where
\begin{align}
J_\mr{LP} = &\int_{\Omega\times \Yf\times \Zf}L(x,y,z)\,d\mu(x,y,z) + \sum_{i=1}^{n_{\mr{b}}}\int_{\partial\Omega\times \Yf\times \Zf} L_{\partial_i}(x,y,z)\,d\mu_{\partial_i}(x,y,z) \nonumber\\ & +\sum_{k=1}^{n_u}\int_{\Omega}L_u^k(x,y) d\nu^k(x,y) + \sum_{i=1}^{n_{\mr{b}}} \sum_{k=1}^{n_{u_i}}\int_{\partial\Omega}L_{u_i}^k(x,y) d\nu_{\partial_i}(x,y)
\end{align}
with $L^k_u$ and $L^k_{u_i}$ denoting the $k^{\mr{th}}$ components of the vector polynomials $L_u$ and $L_{u_i}$ and the convex cone $\Kc$ given by
\begin{align*}
 \Kc = & \; \Mc(\Omega\times \Yf\times \Zf)_+ \times \Mc(\partial\Omega_1\times \Yf\times \Zf)_+\times\ldots\times \Mc(\partial\Omega_{n_{\mr{b}}}\times \Yf\times \Zf)_+ \\
& \times   \; \Mc(\Omega\times \Yf)_+ \times \Mc(\partial\Omega_1\times \Yf)_+\times\ldots\times \Mc(\partial\Omega_{n_{\mr{b}}}\times \Yf)_+ \\
& \times  \;  \Mc(\Omega\times \Yf)_+ \times \Mc(\partial\Omega_1\times \Yf)_+\times\ldots\times \Mc(\partial\Omega_{n_{\mr{b}}}\times \Yf)_+.
\end{align*}

\section{Control design using SDP relaxations}\label{sec:SDPcont}
Since all integrands in the infinite-dimensional LP~(\ref{opt:LPinfCont}) are polynomials and all measures are defined on basic semialgebraic sets, a finite-dimensional SDP relaxation can be constructed following exactly the same procedure as described in Section~\ref{sec:SDPuncont}. The result is the following finite-dimensional SDP:

\begin{equation}\label{opt:SDPcont}
\begin{array}{ll}
\underset{\bs s}{\mbox{inf}} & c_d^\top \bs s  \\
& A_d\bs s = b_d\\
& \bs s \in \Kc_d,
\end{array}
\end{equation}
where
\begin{align*}
\Kc_d = &\; \Mc_d(\Omega\times \Yf\times \Zf)_+ \times \Mc_d(\partial\Omega_1\times \Yf\times \Zf)_+\times \ldots\times \Mc_d(\partial\Omega_{n_{\mr{b}}}\times \Yf\times \Zf)_+ \\
&\times \Mc_d(\Omega\times \Yf)_+ \times \Mc_d(\partial\Omega_1\times \Yf)_+\times \ldots\times \Mc_d(\partial\Omega_{n_{\mr{b}}}\times \Yf)_+ \\
&\times \Mc_d(\Omega\times \Yf)_+ \times \Mc_d(\partial\Omega_1\times \Yf)_+\times \ldots\times \Mc_d(\partial\Omega_{n_{\mr{b}}}\times \Yf)_+
\end{align*}
and where the decision $\bs s$ is partitioned as
\begin{equation}\label{eq:SDPcontVars}
\bs s := [\bs s_\mu^\top ,\;\; \bs s_{\mu_1} ^\top ,\; \ldots ,\;\; \bs s_{\mu_{n_{\mr{b}}}}^\top ,\;\; \bs s_\nu^\top ,\;\; \bs s_{\nu_1} ^\top ,\; \ldots,\; \bs s_{\nu_{n_{\mr{b}}}}^\top,\;\; \bs s_{\hat\nu}^\top ,\;\; \bs s_{\hat \nu_1} ^\top ,\; \ldots,\;\; \bs s_{\hat\nu_{n_{\mr{b}}}}^\top ]^\top,
\end{equation}
corresponding to the truncated moments sequences up to degree $d$ of the respective measures. The optimization problem~(\ref{opt:SDPcont}) is a convex, finite-dimensional, semidefinite program that can be solved by any of the available software such as MOSEK or SeDuMi~\cite{sedumi}. In addition, the matrices and vectors $A_d$, $b$, $c_d$ as well as the cone $\Kc_d$ can be assembled automatically using Gloptipoly 3~\cite{gloptipoly3} starting from a high level description closely resembling the form of the infinite-dimensional LP~(\ref{opt:LPinfCont}).

The following theorem is an exact analogue of~Theorem~\ref{thm:conv}.

\begin{theorem}
Let $p_d$ denote the optimal value of the SDP~(\ref{opt:SDPcont}), let $\hat{p}$ denote the optimal value of the infinite-dimensional LP~(\ref{opt:LPinfCont}) and let $p^\star$ denote the optimal value of the optimal control problem~(\ref{opt:ocp_cont}). Then
\begin{equation}
p_d \le p_{d+1} \le \hat{p}\le p^\star
\end{equation}
for all $d\in\Nb$. In addition,
\begin{equation}
\lim_{d\to\infty}p_d = \hat{p},
\end{equation}
provided that $\Omega$, $\Yf$ and $\Zf$ are compact and the polynomials defining these sets satisfy the Archimedianity condition (see Footnote~\ref{foot:arch}).
\end{theorem} 
\begin{proof}
The proof is a verbatim copy of the proof of Theorem~\ref{thm:conv}, once boundedness of the first components of the moments sequences of $\nu$, $\hat{\nu}$, $\nu_{\partial_i}$, $\hat\nu_{\partial_i}$ is established. But this follows immediately from boundedness of the first component of the moment sequence of $\mu$ and $\mu_{\partial_i}$ and from~(\ref{eq:hat1}) and (\ref{eq:hat2}) with $\phi = 1$.
\end{proof}

\subsection{Controller extraction}
In this section we describe a way to extract a feedback controller from the solution to the SDP~(\ref{opt:SDPcont}).  The idea is to use the fact that the measures $\nu$, $\nu_{\partial_i}$ are absolutely continuous with respect to the measures $\mu$, $\mu_{\partial_i}$ with the density being the control input. The information at our disposal are the truncated approximate moment sequences of these measures obtained from the solution to the SDP~(\ref{opt:SDP}). Let $\bs s^d$ denote this solution for a given even value of $d$, partitioned as in~(\ref{eq:SDPcontVars}). In general, there are a number of ways of obtaining an approximation to a density, given the truncated moment sequences. In our case, we use a simple moment matching technique with a polynomial density.  Therefore, we seek polynomials $\kappa^d(x,y)$ and $\kappa_{\partial_{i}}^d(x,y)$ of degree $d/2$ such that
\begin{subequations}
\begin{align}
\ell_{\bs s_{\nu}^d}(\beta_{d/2}(x,y)) &= \ell_{\bs s_{\mu}}(\beta_{d/2}(x,y)\kappa^d(x,y)) \label{eq:momMatch1aux} \\
\ell_{\bs s_{\nu_{i}}^d}(\beta_{d/2}(x,y)) &= \ell_{\bs s_{\mu_i}}(\beta_{d/2}(x,y)\kappa^d_{\partial_{i}}(x,y)), \label{eq:momMatch2aux}
\end{align}
\end{subequations}
where $\ell(\cdot)$ is the Riesz functional defined in~(\ref{eq:riesz}) and $\beta_{d/2}$ is the vector of monomials in $(x,y)$ up to the total degree $d/2$. Denoting $c_{\kappa}$ and $c_{\kappa_i}$ the coefficients of $\kappa$ and $\kappa^d_{\partial_i}$, we can write $k^d = \beta_{d/2}^\top c_{\kappa}$ and $k^d_{\partial_i} = \beta_{d/2}^\top c_{\kappa_i}$. Inserting these into~(\ref{eq:momMatch1aux}) and (\ref{eq:momMatch2aux}) leads to
\begin{subequations}
\begin{align}
\tilde{\bs s}_{\nu}^d &= M(\bs s_{\mu}^d)c_{\kappa} \label{eq:momMatch1} \\
\tilde{\bs s}_{\nu_i}^d &= M(\bs s_{\mu_i}^d)c_{\kappa_i} \label{eq:momMatch2},
\end{align}
\end{subequations}
\looseness-1 where $\tilde{\bs s}_{\nu}^d$ respectively $\tilde{\bs s}_{\nu_i}^d$ denote the first $\binom{n+d/2}{n}$ components of $\bs s_{\nu}^d$ respectively $\bs s_{\nu_i}^d$ (i.e., all moments up to degree $d/2$) and $M(\bs s_\mu^d)$ and $M(\bs s_{\mu_i}^d)$ denote the moment matrices associated to $\bs s_{\mu}^d$ and $\bs s_{\mu_i}^d$ defined in~(\ref{eq:momMatOnlyDef}). The equations (\ref{eq:momMatch1}) and (\ref{eq:momMatch2}) are a linear system of equations that can be solved for $c_k$ and $c_{k_i}$  (exactly if the moment matrices are invertible or in a least-squares sense otherwise), thereby obtaining the polynomial feedback controllers $\kappa^d$ and $\kappa^d_{\partial_i}$. The control inputs $u(x)$ and $u_{\partial_i}(x)$ for the PDE~(\ref{eq:pdeEntireCont}) are then given in a feedback form by
\begin{subequations}
\begin{align}
u(x) &= \kappa^d(x,y(x)), \label{eq:contPolyDist}\\ 
u_{\partial_i}(x) &= \kappa^d_{\partial_i}(x,y(x)). \label{eq:contPolyBnd}
\end{align}
\end{subequations}

\section{Higher-order PDEs}\label{sec:higherOrder}
In this section we briefly describe how the presented approach extends to higher-order PDEs. We do not present a general theory but rather demonstrate the approach on an example. The idea is to introduce additional variables for higher order derivatives. Consider, for example,  the Dym equation~\cite{strauss1990nonlinear}
\[
\frac{\partial y_1}{\partial x_1}	- y_1^3\frac{\partial^3 y_1}{\partial x_2^3} = 0,
\]
where $x_1$ typically represents the time, $x_2$ is the spatial variable and $y_1(x_1,x_2)$ is the unknown function. We introduce the variables \[
y_2  = \frac{\partial y_1}{\partial x_2},\quad y_3  = \frac{\partial y_2}{\partial x_2}
\]
with which we obtain
\begin{align*}
\frac{\partial y_1}{\partial x_1}	- y^3 \frac{\partial y_2}{\partial x} &= 0, \\
y_2  - \frac{\partial y_1}{\partial x_2} &= 0, \\
y_3  - \frac{\partial y_2}{\partial x_2}  & = 0.
\end{align*}
This is a system of PDEs of the form~(\ref{eq:pde}) with
\[
F(x,y,\Dc y) = \begin{bmatrix} \frac{\partial y_1}{\partial x_1}	- y^3 \frac{\partial y_2}{\partial x_2} \\ y_2  - \frac{\partial y_1}{\partial x_2} \\ y_3  - \frac{\partial y_2}{\partial x_2} \end{bmatrix},\quad B_{i,j} = 0,
\]
where $y = (y_1,y_2,y_3)$ and $x=(x_1,x_2)$. From here, the approach proceeds as described in the previous sections. Note, however, that in this case, most of the derivatives of $y$ appear linearly in $F$ and therefore the computational complexity of the approach can be significantly reduced using the method described in Section~\ref{sec:complexRedLinearDer}.

\section{Complexity analysis $\&$ reduction}\label{sec:complex}
In this section we briefly discuss the computational complexity of the presented approach and describe several means of complexity reduction.

The asymptotic complexity of solving the SDPs (\ref{opt:SDP}) or (\ref{opt:SDPcont}) is governed by the dimension of the sets that the measures in the LPs (\ref{opt:LPinf}) or (\ref{opt:LPinfCont}) are supported on and by the relaxation degree $d$. This dimension is the same for the uncontrolled and controlled case and therefore we focus our discussion on the former, the conclusions being exactly the same for the latter. Without any further structure of the PDE~(\ref{eq:pdeEntire}), the measures appearing in the LP~(\ref{opt:LPinf}) are supported on sets of dimension 
\begin{equation}\label{eq:nvar}
n_{\mr{var}} = n + n_y + n_y\cdot n\,.
\end{equation}
The dimension of the largest cone $\Mc_d(\cdot)$ defining $\Kc_d$ in~(\ref{opt:SDP}) is then $N(N+1)/2$, where
\[
N = \binom{n_{\mr{var}} + d/2}{n_{\mr{var}}}.
\]
Assuming no structure is exploited and given the current state of the art of interior point semidefinite programming solvers (i.e., MOSEK), the number $N$ should not exceed $\approx 1000$. The value of $N$ for different values of $n_{\mr{var}}$ and $d$ are summarized in Table~\ref{tab:complex}. We note that having $N\le 1000$ does not guarantee that the SDP problem can be accurately solved; this depends also on the number of equality constraints (i.e., the number of rows of $A_d$) as well as on the numerical conditioning of the problem.

\begin{table}[h]
\centering
\caption{\small \rm The value of $N$ for different values of $n_{\mr{var}}$ and $d$. The last three columns corresponds to the values for the incompressible Navier-Stokes equations in three spatial dimensions with and without complexity reduction techniques described in Section~\ref{sec:complexRedLinearDer}.}\label{tab:complex}\vspace{2mm}
{\small
\begin{tabular}{c|cccccccc}
\toprule
$n_{\mr{var}}$&         4    &       6       &     8      &       10     &    12     &    17       &       21            & 24        \\\midrule
 $d = 4$  &               15   &      28     &      45     &      66     &     91    &     171    &       253          & 325     \\
  $d = 6$  &              35   &      84     &     165    &    286     &   455    &    1140   &       2024        & 2925          \\    
   $d = 8$  &             70   &     210    &    495     &    1001   & 1820    &    5985   &       12650      & 20475          \\
\bottomrule
\end{tabular}
}
\end{table}

\paragraph{First order methods} We note that, for larger values of  $N$ (say $N \le 5000$), one can resort to first order optimization methods for semidefinite programming, the state of the art being the augmented Lagrangian-based solver SDPNAL+~\cite{sdpnalPLus}. These solvers, however, often do not achieve the accuracy of interior point solvers and in general struggle more with numerical conditioning, which is an issue for SDP relaxations formulated in the monomial basis, as is the case here. Therefore, if these solvers are to be successfully used for the problems studied here, a different  basis may be required (the choice of which is not trivial in a multivariate setting), or an iterative preconditioning applied,  similarly to~\cite{nie2012regularization}.

\subsection{Complexity reduction}\label{sec:complexRed}
In view of Table~\ref{tab:complex} and the present status of state-of-the-art SDP solvers, the reader can realize how crucial is the number of variables 
$(x,y,z)$ when implementing the SDP relaxations described earlier. 
For instance, in some cases, even reducing the number of variables
by only one, may allow for implementing an additional
step in the computationally expensive SDP hierarchy. 
We briefly describe in Section~\ref{sec:complexRedLinearDer} cases where one may reduce the number of variables.

Another route briefly described in Section~\ref{sec:sparsAndDeg} (possibly combined with reduction of variables)
is to exploit  sparsity coming from the weak coupling between the variables $(x,y)$ and~$z$. Indeed such a weak coupling has already been successfully exploited in SDP hierarchies for polynomial optimization (e.g., \cite{lasserre2006convergent,waki2006sums}). Also the fact that the variable $z$ has a low degree (1 or 2) in the problem description can be exploited.


\subsubsection{Derivatives appearing linearly}\label{sec:complexRedLinearDer}
The first opportunity for complexity reduction shows itself when some of the derivatives appear only linearly in the PDE considered and can be solved for explicitly. In that case, the complexity can be reduced by introducing additional variables $z_{i,j} = \frac{\partial y_i}{\partial x_j}$ only for those derivatives that appear non-linearly. We do not develop a general framework for this but rather demonstrate the approach on an example. Consider the Burgers' equation (e.g., \cite{evansPDEbook})
\[
\frac{\partial{y}}{\partial x_1} + y\frac{\partial{y}}{\partial x_2}  = 0,
\]
where $x_1$ typically represents the time and $x_2$ the space. The only derivative appearing nonlinearly is $ \frac{\partial{y}}{\partial x_2}$, whereas $ \frac{\partial{y}}{\partial x_1}$ appears linearly and can be expressed as $ \frac{\partial{y}}{\partial x_1} = -y\frac{\partial{y}}{\partial x_2}$. Setting $z = \frac{\partial{y}}{\partial x_2}$ we obtain the system of PDEs
\begin{align*}
\frac{\partial{y}}{\partial x_1} + yz  = 0, \\
 \frac{\partial{y}}{\partial x_2} -  z = 0.
\end{align*}
Given any test function $\phi \in C^\infty(\Omega \times \Yf
)$, we get
\[
\
\int_{\partial\Omega} \phi(x,y(x))\eta_1(x)\,d\sigma(x)  = \int_{\Omega} \frac{d}{d_{x_1}}\phi(x,y(x))\,dx = \int_{\Omega} \frac{\partial \phi}{\partial x_1} +   \frac{\partial \phi}{\partial y}\frac{\partial y}{\partial x_1} \,dx = \int_{\Omega} \frac{\partial \phi}{\partial x_1} - \frac{\partial \phi}{\partial y}yz \,dx 
\]
and
\[
\int_{\partial\Omega} \phi(x,y(x))\eta_2(x)\,d\sigma(x)  = \int_{\Omega} \frac{d}{d_{x_2}}\phi(x,y(x))\,dx = \int_{\Omega} \frac{\partial \phi}{\partial x_2} +   \frac{\partial \phi}{\partial y}\frac{\partial y}{\partial x_2} \,dx = \int_{\Omega} \frac{\partial \phi}{\partial x_2} + \frac{\partial \phi}{\partial y}z \,dx.
\]
Applying~(\ref{eq:occupMeasFun}) and (\ref{eq:boundaryMeasFun}), we get
\begin{subequations}
\begin{equation}
\int_{\partial\Omega \times \Yf \times \Zf} \phi(x,y)\eta_1(x)\,d\mu_{\partial}(x,y,z) = \int_{\Omega \times \Yf\times \Zf} \frac{\partial \phi}{\partial x_1} - \frac{\partial \phi}{\partial y}yz \,d\mu(x,y,z) \label{eq:red1}
\end{equation}
\begin{equation}
\int_{\partial\Omega\times \Yf \times \Zf} \phi(x,y)\eta_2(x)\,d\mu_{\partial}(x,y,z)  = \int_{\Omega\times \Yf\times \Zf} \frac{\partial \phi}{\partial x_2} + \frac{\partial \phi}{\partial y}z \,d\mu(x,y,z). \label{eq:red2}
\end{equation}
\end{subequations}
These equations replace equations~(\ref{eq:thmEq1}) and (\ref{eq:thmEq2}). Importantly, there are four variables appearing in the measures $\mu$ and $\mu_\partial$ in~(\ref{eq:red1}) and (\ref{eq:red2}), which is one less compared to when (\ref{eq:thmEq1}) and (\ref{eq:thmEq2}) are used to linearly represent this PDE. In addition, provided that the boundary function $G$ in~(\ref{eq:bnd}) does not depend on the derivatives of $y$, the variable $z$ can be removed from $\mu_\partial$, further reducing the complexity.

 This process can be carried out in a much more general setting. For example, using the complexity reduction described here, the total number of variables for the incompressible Navier-Stokes equations in one temporal and three spatial dimensions is~17, compared to 24 that would be required using the general formulation. Indeed, we have $n_y = 4$ (three components of the velocity field and the scalar pressure) and $n = 4$ (one temporal plus three spatial dimensions), leading to $n_{\mr{var}} = 24$ in~(\ref{eq:nvar}). Using the fact that the temporal derivatives of the velocity components appear linearly and can be explicitly solved for leads to the reduction to $n_{\mr{var}} = 21$. In addition, choosing the test functions $\phi$ to be independent of the pressure allows for a further reduction to $n_{\mr{var}} = 17$ (but possibly at the cost of looser relaxation).  These reductions translate to a significant decrease in the size of the SDP matrices $N$ as documented by the last two columns of Table~\ref{tab:complex}.

\subsubsection{Sparsity $\&$ degree bounding}\label{sec:sparsAndDeg}
Another way to reduce complexity is to exploit structure of the equations~(\ref{eq:thmEq1})-(\ref{eq:thmEq3}) (the same considerations hold for the controlled case (\ref{eq:thmEq1Cont})-(\ref{eq:thmEq3Cont})). The main source of structure to be exploited is the way that the variable $z$ enters the equations.

\paragraph{Degree bounding} First, notice that the degree of $z$ is at most two (and equal to one if $B_{i,j} = 0$). Therefore a natural way to reduce complexity of the SDP relaxation is to bound the degree of the moments corresponding to the monomials containing $z$. That is, we consider a hierarchy indexed by $d$ and $\tilde{d}$, where $\tilde d \ge 2$ is either fixed to a constant value or a fraction of $d$; the moments that appear in the reduced-degree SDP hierarchy are then
\[
\int_{\Omega\times\Yf\times\Zf} x^\alpha y^\beta z^\gamma \, d\mu(x,y,z)
\]
such that $|(\alpha,\beta,\gamma)| \le d$ and $|\gamma| \le \tilde d$ (and analogously for $\mu_{\partial_i}$). Fixing $\tilde{d}$ to a small value, e.g., $2$ or $4$, will dramatically reduce the size of the SDP matrices $N$, especially in those situations with a large dimension of $z$ (i.e., when there is a large number of the derivatives of $y$ appearing nonlinearly in the PDE~(\ref{eq:pdeEntire}) or (\ref{eq:pdeEntireCont})). Note, however, that this complexity reduction comes at a cost of potentially worsening the bounds obtained from solving~(\ref{opt:SDP}) or the controllers obtained from~(\ref{opt:SDPcont}). Nevertheless, as long as both $d$ and $\tilde d$ tend to infinity, the convergence results of Theorem~\ref{thm:conv} are preserved.

\paragraph{Sparsity}
Another possibility to reduce the computational complexity is to exploit the fact that the variable $z$ enters the equations~(\ref{eq:thmEq1})-(\ref{eq:thmEq3}) in a way that lends itself to the sparse SDP relaxation~\cite{waki2006sums}. Indeed, as long as the degree of $F$ and $G$ in the $z$ variable is at most two, all monomials appearing in these equations contain at most the products of two components of~$z$ (i.e., monomials of the form $x^\alpha y^\beta z_{i} z_{j}$); here we assume for notational simplicity a linear indexing of the matrix variable $z$. Hence, one can use the sparse SDP relaxation of \cite{waki2006sums}, which in our case of a moment problem (rather than the dual sum-of-squares problem consider in~\cite{waki2006sums}) translates to introducing measures $\tilde \mu_{i,j}$ supported on the variables $(x, y, z_i, z_j)$ and imposing the consistency constraints requiring that the $(x,y,z_i)$ marginals of the measures coincide for all $i \le n_z$, where $n_z$ is the dimension of the $z$ variable (equal to $n_y\cdot n$ if all derivatives of $y$ appear nonlinearly in~(\ref{eq:pdeEntire})). This procedure leads to a dramatic reduction in the size of the SDP matrices to
\[
N_{\mr{sparse}} = \binom{\tilde n_{\mr{var}} + d/2}{\tilde n_{\mr{var}}},
\]
where 
\[
\tilde n_{\mr{var}} = n + n_y + 2,
\]
which is tractable for most computational physics applications. We note that the number of measures introduced in this way is $n_z(n_z-1)/2$, leading to the same increase in the number of blocks of the SDP cone $\Kc_d$ in~(\ref{opt:SDP}). However, in most situations, the reduction of the block size achieved in this way is far more significant in terms of computational complexity than the increase in the number of these blocks.


Convergence results of Theorem~\ref{thm:conv} are preserved as long as the the so-called running intersection property~\cite{lasserre2006convergent} is satisfied. See also~\cite{nie2008sparse} for a discussion of computational complexity of different sparse relaxation techniques.

\section{Numerical examples}\label{sec:numEx}
\subsection{Burgers' equation -- analysis}\label{sec:burgAnal}
We consider the Burgers' equation (e.g., \cite{evansPDEbook})
\begin{equation}\label{eq:burgUncont}
\frac{\partial{y}}{\partial x_1} + y\frac{\partial{y}}{\partial x_2}  = 0,
\end{equation}
where $x_1$ typically represents the time and $x_2$ is the spatial variable.  The goal is to find bounds on the functional~(\ref{eq:obj}) evaluated along the solutions of the PDE~(\ref{eq:burgUncont}). The domain  is given by $\Omega = [0,T]\times [0,L]$ with $T = 5$ and $L = 1$. For numerical reasons, we scaled the $x_1$ variable such that $x_1 \in [0,1]$ (when working with the monomial basis, we always advise to scale all variables to unit boxes or balls). We consider a periodic boundary condition on the spatial variable $x_2$, i.e., $y(x_1,0) = y(x_1,L)$ for all $x_1\in [0,T]$; this boundary condition is imposed as described in Section~\ref{sec:periodicBnd}. We also impose the initial condition $y(0,x_2) = y_0(x_2) =  10(x_2(1 - x_2))^2$; this Dirichlet-type initial condition is imposed as described in Section~\ref{sec:nonPolBnd} (despite being polynomial). For the objective functional we first choose $L(x,y) = y^2$ and $L_\partial(x,y) = 0$ in~(\ref{eq:obj}), i.e., we are seeking bounds on  $\int_{\Omega} y(x)^2\,dx$, which can be interpreted as the average kinetic energy of the solution. It is straightforward to prove that in this case the moments $a_k(x_1) =  \int_0^L y^k(x_1,x_2)\,dx_2$ are preserved (i.e., $\frac{d}{d_{x_1}}a_k(x_1) = 0$) for the solution of the PDE~(\ref{eq:burgUncont}). Therefore
\[
\int_{[0,T]\times [0,L]} \hspace{-10mm} y(x)^2\,dx = T a_2(0) = \int_0^L \hspace{-2mm} y^2(0,x_2)\,dx_2 = 5\int_{0}^1 [10(x_2(1 - x_2))^2]^2\,dx_2 = \frac{50}{63} \approx 0.79365079365.
\]

Solving the SDP relaxation~(\ref{opt:SDP}) for $d = 4$ results in 
\[
\hat{p}_{\mr{inf},d}^{\mr{SDP}} \approx 0.79365079357, \quad \hat{p}_{\mr{sup},d}^{\mr{SDP}} \approx0.79365080188,
\]
i.e., achieving precision of nine respectively six significant digits for the lower and upper bounds. However, this extremely high accuracy is due to the special special choice of the objective functional along with the invariance of $a_2(x_1)$ and cannot be expected in general. Indeed, the invariance condition implies that $\bs s_{(0,2)} = T \int_{0}^L y_0(x_2)\,dx_2 (=50/63)$ is among the constraints of the SDP~(\ref{opt:SDP}) and hence the objective functional $c^\top \bs s = \bs s_{(0,2)}$ is fully specified by equality constraints only.

Next, we consider an objective functional where the value of the objective functional is not fully determined by the constraints. We choose $L(x) = x_2^2y^2$ and $L_{\partial} = 0$, i.e., we want to find bounds on $\int_{[0,T]\times [0,L]} x_2^2 y^2(x_1,x_2)\, dx_1dx_2  $. The results are summarized in Table~\ref{tab:bndsx2y2}. We have set $\Yf = \Rb$ and $\Zf = \Rb$, since we assume to have no a priori bound on the supremum of $|y|$ and since $\frac{\partial y}{\partial x_2}$ is known to be discontinuous in this case; therefore, the convergence result of Theorem~\ref{thm:conv} does not apply in this case. Nevertheless we observe an increasing accuracy of the bounds.

\begin{table}[h]
\centering
\caption{\small \rm Uncontrolled Burgers equation: Upper and lower bounds on the functional~(\ref{eq:obj}) with $L= x_2^2y^2$ and $L_\partial = 0$.}\label{tab:bndsx2y2}\vspace{2mm}
{\small
\begin{tabular}{ccccccc}
\toprule
$d$                            &         4        &       6            &       8              \\\midrule
  Lower bound (SeDuMi)  &      $0.206$   &  $0.263$    &           $0.276$   \\
    Upper bound (SeDuMi)  &    $0.380$   &  $0.297$    &           $0.283$      \\\midrule
    Parse time (Gloptipoly 3)  & 2.91s & 3.41 s & 6.23 s \\ 
          SDP solve time (SeDuMi / MOSEK)   &   $2.62\, /\, 1.63 \,\mr{s}$    &  $2.61 \, / \, 1.32 \,\mr{s}$     &           $20.67 \, / \, 7.05^\star\,\mr{s}$        \\
\bottomrule
\end{tabular}
}
\end{table}

\renewcommand\thefootnote{}
\footnotetext{$^\star$ For $d = 8$ the MOSEK solver reported an ill posed dual problem even though the objective value converged to a valid bound (up to a numerical roundoff error). This may be caused by the absence of the Slater condition. The investigation of the dual SDP, its theoretical interpretation, practical use and numerical properties are left for future work.}
\renewcommand\thefootnote{\arabic{footnote}}

\subsection{Burgers' equation -- control}\label{sec:burgCont}
Now we turn to controlling the Burgers' equation. We consider distributed control, i.e.,
 \begin{equation}\label{eq:burgCont}
\frac{\partial{y}}{\partial x_1} + y\frac{\partial{y}}{\partial x_2}  = u(x_1,x_2),
\end{equation}
where $u(\cdot,\cdot)$ is the control input to be designed subject to the constraint $u(x_1,x_2) \in [-1,1]$. The domain is given by $\Omega = [0,T]\times [0,L]$ with $T = 3$ and $L = 1$ and we consider a periodic boundary constraint on the $x_2$ variable, enforced as described Section~\ref{sec:burgAnal}. The control goal is the minimization of the energy of the solution  $\int_{[0,T]\times [0,L]} y(x_1,x_2)^2\, dx$ starting from the initial condition $y(0,x_2) = y_0(x_2) = 10(x_2(1 - x_2))^2$. This is enforced by selecting $L(x,y) = y^2$ and $L_{\partial} = L_{u}^k = L_{u_i}^k = 0$  in the objective functional~(\ref{eq:objCont}); the initial condition is imposed as described in Section~\ref{sec:nonPolBnd}. We solve~(\ref{opt:SDPcont}) with $d = 6$, where we used the complexity reduction method described in Section~\ref{sec:complexRedLinearDer} to derive the SDP.  From the solution to the SDP, we extracted a polynomial controller of degree three of the form~(\ref{eq:contPolyDist}) using the moment matching equation~(\ref{eq:momMatch1}). For the closed-loop simulation, the PDE was discretized using a finite difference scheme with time step (i.e., for $x_1$) of $0.01\,\mr{s}$ and spatial discretization of the $x_2$ variable with 100 grid points. We observe a stable behavior with the the solution converging to zero as $x_1$ (= time) increases.

\begin{figure}[h!]

    	\begin{picture}(50,160)
	\put(15,-20){\includegraphics[width=75mm]{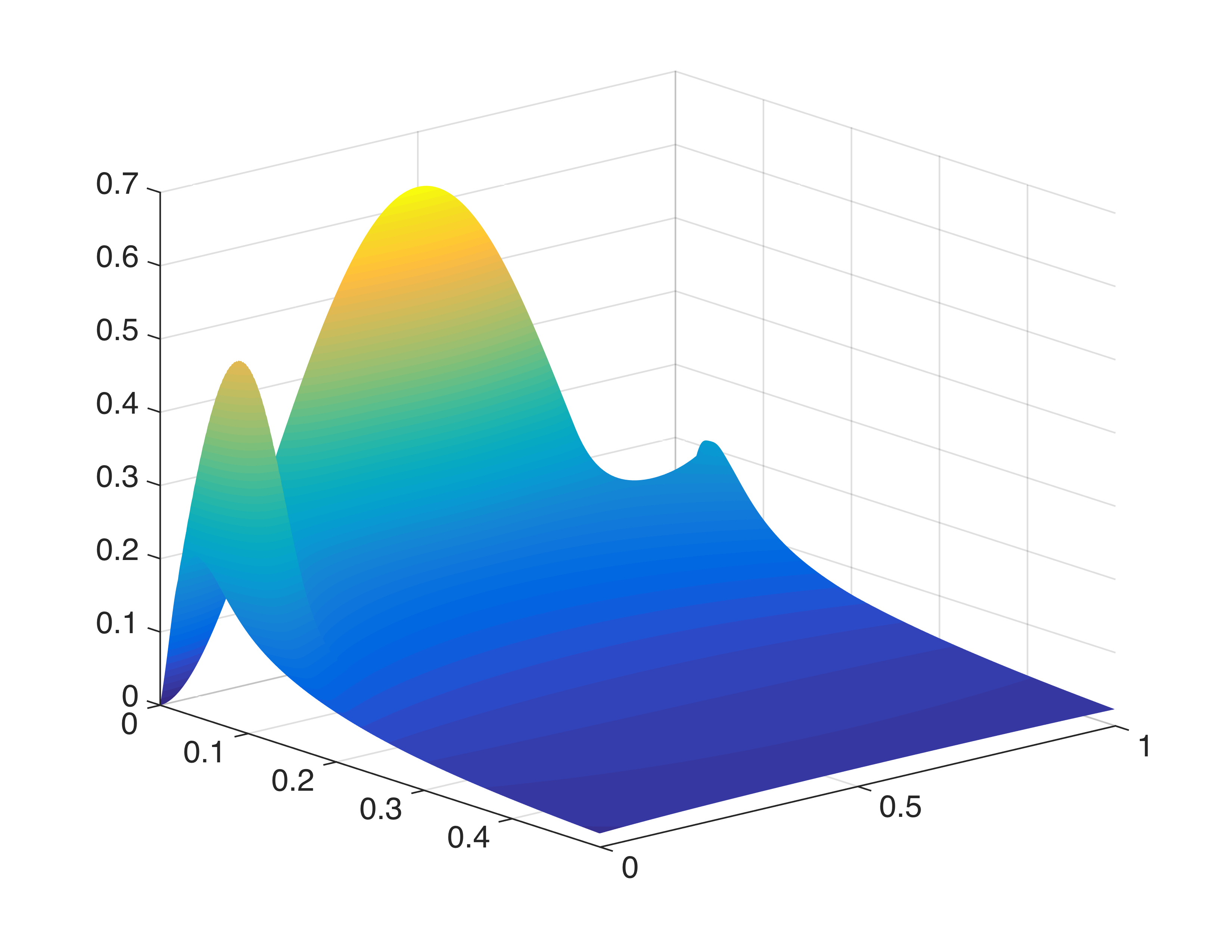}}
			\put(250,-20){\includegraphics[width=80mm]{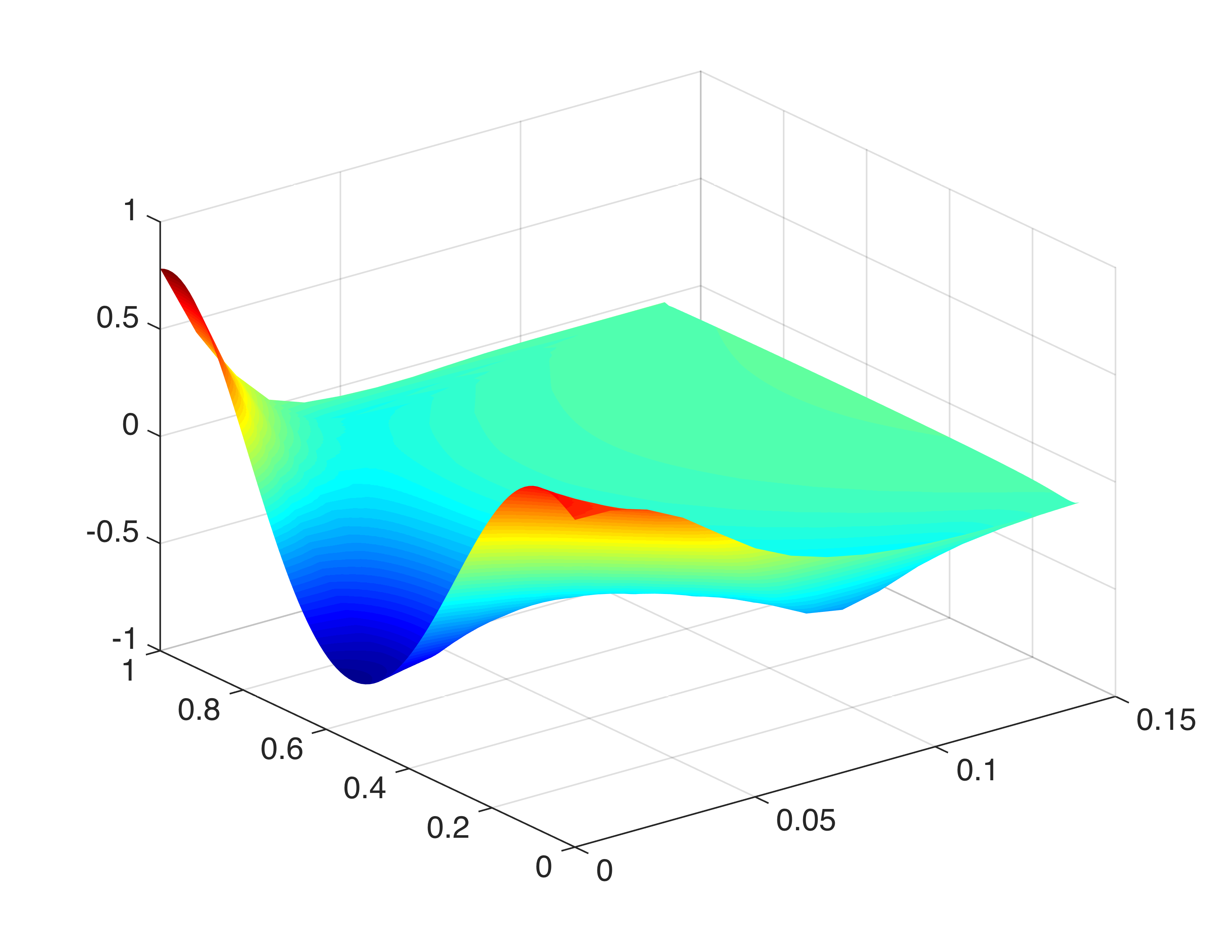}}

	\put(40,0){\rotatebox{-12}{\footnotesize{$x_1$ (=time)}}}
	\put(175,-5){\footnotesize{$x_2$}}
	\put(-13,60){\footnotesize{$y(x_1,x_2)$}}

	\put(290,5){\footnotesize{$x_2$}}
	\put(410,-4){\rotatebox{11}{\footnotesize{$x_1$ (=time)}}}
	\put(222,70){\footnotesize{$u(x_1,x_2)$}}
	\end{picture}

    \caption{\footnotesize Control of Burgers' equation~(\ref{eq:burgCont}). The control input $u(x_1,x_2)$ is given by $u(x_1,x_2) =\kappa^d(x_1,x_2,y(x_1,x_2))$, where $\kappa^d$ is a polynomial of degree three in $(x_1,x_2,y)$. Time to solve the SDP~(\ref{opt:SDPcont}) was $\approx$ 1.2\,s, using Matlab + MOSEK 8, 2$\,$GHz intel i7. Evaluation  of the polynomial feedback control law for each time step on the entire spatial grid of 100 points takes approximately $ 1.5\,\mr{ms}$.
   } \label{fig:burgerMeas}
\end{figure}
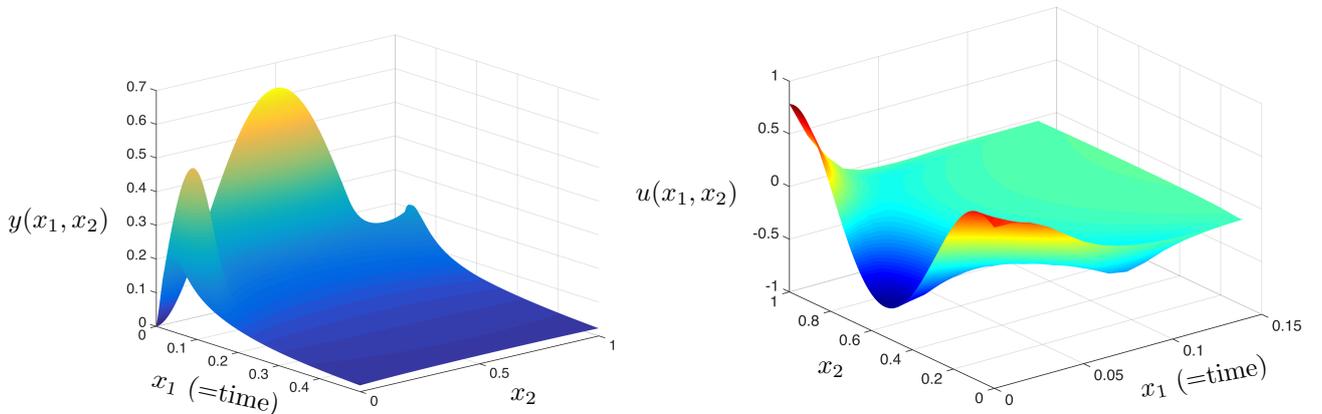

\section{Conclusion}\label{sec:conclusion}
We have presented a convex-optimization-based method for analysis and control of nonlinear partial differential equations (PDEs). The method proceeds by embedding the problem at hand to an infinite-dimensional space of Borel measures, with the so called occupation and boundary measures as variables. This leads to an infinite-dimensional linear programming problem (LP), which is  then approximated by a sequence of finite-dimensional convex semidefinite programming problems (SDPs), with proven convergence to the optimal value of the LP. The solutions to these SDPs provide either bounds on (possibly nonlinear) functionals of solutions to the PDE in the uncontrolled case, or polynomial feedback controllers in the controlled case. The major advantage of the approach is its overall convexity and the absence of spatio-temporal gridding. We have discussed computational complexity of the approach and outlined several ways of complexity reduction, exploiting the structure of the problem.

There are several open questions left for future research. The first one is the equivalence of the infinite-dimensional LP to the original problem. The starting point here may be the theory of weak-strong solutions~\cite{brenier2011weak}, originated in the work~\cite{lions1996mathematical}. In these works, additional convex constraints are introduced, ensuring the uniqueness of the weak (i.e., measure-valued) solution whenever the strong solution is unique. These so-called entropy constraints could be added to the infinite-dimensional LP considered here, thereby proving the equivalence, at least in the uncontrolled case and for the class of Euler PDEs considered in~\cite{brenier2011weak}. The situation in the more general setting considered in this work is more complex and remains an open question, especially in the controlled case.

Another question is that of computational complexity reduction. We have outlined several of the most obvious ways of doing so in Section~\ref{sec:complexRed}. Nevetheless, we believe that a more detailed inspection of the problem structure (for example, for a particular class of PDEs studied) would lead to further reduction. Numerical conditioning of the resulting SDPs as a function of the problem data is also an important and not well understood question. Finally, a problem-adapted basis is very likely to be superior to the generic monomial basis used here in terms of numerical conditioning, thereby allowing the SDP relaxations to be accurately computed for higher values of $d$, possibly with first-order optimization methods that scale much more favorably with the problem size but struggle with numerical conditioning.

\section{Acknowledgments}
The first author would like to thank Hassan Arbabi for kindly providing the numerical solver for the Burgers' equation as well as for an insightful discussion and comments on the paper.

The authors would also like to thank Swann Marx for pointing out the work~\cite{diperna1985measure} and a stimulating discussion on the topic. This work also benefited from interesting discussions with Martin Kru\v z\'{\i}k and Josef M\'alek.

The research of M. Korda was supported by the Swiss National Science Foundation under grant P2ELP2\_165166. The research of D. Henrion and J.B. Lasserre was funded by the European Research Council (ERC) under the European’s Union Horizon 2020 research and innovation program (grant agreement 666981 TAMING).

\bibliographystyle{abbrv}

\end{document}